\documentclass[reqno,12pt]{amsart}
\usepackage{a4wide}
\usepackage{amsmath}
\usepackage{amssymb}
\usepackage{amsthm}

\numberwithin{equation}{section}

\newtheorem{theo}{Theorem}

\newtheorem{coro}{Corollary}

\newtheorem{lem}{Lemma}

\theoremstyle{remark}
\newtheorem{Remark}{Remark}
\newtheorem{Remarks}[Remark]{Remarks}

\def\fl#1{\left\lfloor#1\right\rfloor}

\def\v#1{\left\vert#1\right\vert}
\def\R{\mathbb R}
\def\bz{{\mathbf z}}
\def\bv{{\mathbf v}}
\def\bu{{\mathbf u}}
\def\bz{{\mathbf z}}
\def\ba{{\mathbf a}}
\def\bi{{\mathbf i}}
\def\bj{{\mathbf j}}
\def\bk{{\mathbf k}}
\def\bm{{\mathbf m}}
\def\bM{{\mathbf M}}
\def\bn{{\mathbf n}}
\def\bK{{\mathbf K}}
\def\bL{{\mathbf L}}
\def\bN{{\mathbf N}}
\def\bP{{\mathbf P}}

\def\bid{{\mathbf 1}}
\def\bnull{{\mathbf 0}}
\def\Z{\mathbb Z}

\def\al{\alpha}
\def\be{\beta}

\def\ep{\varepsilon}

\def\la{\lambda}

\def\Ga{\Gamma}

\def\({\left(}
\def\){\right)}
\def\[{\left[}
\def\]{\right]}

\def\fl#1{\left\lfloor#1\right\rfloor}

\def\dd{\partial}

\begin{document}

\title[]{Multivariate $p$-adic formal congruences and 
integrality of Taylor coefficients of mirror maps}

\author[]{C. Krattenthaler$^\dagger$ and T. Rivoal}
\date{\today}

\address{C. Krattenthaler, Fakult\"at f\"ur Mathematik, Universit\"at Wien,
Nordbergstra{\ss}e~15, A-1090 Vienna, Austria.
WWW: \tt http://www.mat.univie.ac.at/\~{}kratt.}

\address{T. Rivoal,
Institut Fourier,
CNRS UMR 5582, Universit{\'e} Grenoble 1,
100 rue des Maths, BP~74,
38402 Saint-Martin d'H{\`e}res cedex,
France.\newline
WWW: \tt http://www-fourier.ujf-grenoble.fr/\~{}rivoal.}

\thanks{$^\dagger$Research partially supported 
by the Austrian Science Foundation FWF, grants Z130-N13 and S9607-N13,
the latter in the framework of the National Research Network
``Analytic Combinatorics and Probabilistic Number Theory''}

\thanks{This paper was written in part during the authors' stay 
at the Erwin Schr\"odinger Institute for Physics and Mathematics,
Vienna, during the programme ``Combinatorics and Statistical Physics''
in Spring 2008.}



\begin{abstract}
We generalise Dwork's theory of $p$-adic formal congruences 
from the univariate to a multi-variate setting. We apply our
results to prove integrality assertions on the Taylor coefficients of
(multi-variable) mirror maps. More precisely, 
with $\bz=(z_1,z_2,\dots,z_d)$,
we show that the Taylor coefficients of the multi-variable series
$q(\bz)=z_i\exp(G(\bz)/F(\bz))$ are integers, where 
$F(\bz)$ and 
$G(\bz)+\log(z_i) F(\bz)$, $i=1,2,\dots,d$, 
are specific solutions of certain GKZ systems.
This result implies the integrality
of the Taylor coefficients of numerous families of multi-variable mirror maps 
of Calabi--Yau complete intersections in weighted projective spaces,
as well as of many one-variable mirror maps in the ``{\em Tables of 
Calabi--Yau equations}'' [{\tt ar$\chi$iv:math/0507430}] of Almkvist,
van Enckevort, van Straten and Zudilin.
In particular, our results prove a conjecture of Batyrev and van Straten in
[{\em Comm. Math. Phys.} {\bf 168} (1995), 493--533] on the
integrality of the Taylor coefficients of canonical coordinates for
a large family of such coordinates in several variables.
\end{abstract}

\maketitle

\section{Introduction and statement of the results}

In \cite{dworkihes,dwork1,dwork2,dwork3,dwork}, Dwork developed a
sophisticated theory for proving 
analytic and arithmetic properties of solutions to ($p$-adic)
differential equations. In \cite{dworkihes,dwork}, he focussed on the
case of hypergeometric differential equations. In particular, the
article \cite{dwork} 
contains a ``formal congruence'' criterion that enabled him to
address the analytic continuation of quotients of certain solutions
and to establish arithmetic properties satisfied by exponentials of
such quotients. These exponentials of ratios of solutions to
hypergeometric differential equations (in fact, of Picard--Fuchs
equations) have recently received great attention in mathematical
physics and algebraic geometry under the name of {\it canonical
coordinates}. Their compositional inverses, known as 
{\it mirror maps}, are an important
ingredient in the computation of the Yukawa coupling in the
theory of mirror symmetry. It is conjectured that the coefficients in
the Lambert series expansion of the Yukawa coupling produce Gromov--Witten
invariants of classes of rational curves.

It is only relatively recent, that Dwork's theory has been
systematically applied to obtain general arithmetic results on the
Taylor coefficients of mirror maps. 
Partial results in this direction were found by Lian and Yau
\cite{lianyau,lianyau2}, by Zudilin 
\cite{zud}, and by Kontsevich, Schwarz and Vologodsky \cite{KoSVAA,volog}. 
The (so far) strongest
and most general results are contained in~\cite{dela,kr,kr2},
where, in particular, numerous integrality
results for the Taylor coefficients of univariate mirror maps 
of Calabi--Yau complete intersections in weighted projective spaces
are proven, 
improving and refining the afore-mentioned results by Lian and 
Yau, and by Zudilin. However, all these results do not touch the case
of {\it multi-variable} mirror maps, upon which they are not able to
say anything. The goal of this paper is to set the basis of a
theory which is capable to address questions of integrality of Taylor
coefficients of multi-variable mirror maps,
and to apply this theory
systematically to large classes of
such mirror maps.

\subsection{Multivariate theory of formal congruences}\label{ssec:0}

The proof strategy in \cite{dela,kr,kr2,lianyau,lianyau2,zud} for obtaining
integrality assertions on the Taylor coefficients of one-variable
mirror maps is crucially
based on a series of reductions and results, of which the corner stones
are: 

\begin{enumerate} 
\item[(D1)]
the conversion of the integrality problem to a $p$-adic problem;
\item[(D2)] a lemma due to Dieudonn\'e and Dwork (cf.\ \cite[Ch.~14, p.~76]{lang})
providing a criterion for
deciding whether a power series with coefficients over $\mathbb Q_p$
has coefficients in $\mathbb Z_p$;
\item[(D3)] a reduction lemma for harmonic numbers due to the authors
(cf.\ \cite[Lemma~1, respectively Lemma~5]{kr} and
\cite[Lemma~3]{kr2});
\item[(D4)] a combinatorial lemma due to Dwork \cite[Lemma~4.2]{dwork} 
for rearranging sums that appear in this context in a way tailor-made
for $p$-adic analysis;
\item[(D5)] Dwork's theorem on formal congruences (cf.\ \cite[Theorem
1.1]{dwork}).
\end{enumerate}

We point out that Lian and Yau, and Zudilin do not need item (D3) due
to the nature of the special families of mirror maps that they were
considering. Indeed, item (D3) is the decisive novelty which enabled
the authors to arrive at their general sets of results in
\cite{kr,kr2}. On the side, we remark that Zudilin also condenses (D4)
and (D5) into one step in the proof of his main result in
\cite{zud}. However, in order to arrive at the general results in
\cite{kr,kr2}, it turned out to be necessary to follow the full path
outlined by (D1)--(D5) above, as attempts to lift Zudilin's variation
to this generality failed. 

With the exception of (D1), which trivially extends to the
multi-variable case, for none of the above items there exist
multi-variate extensions in the current literature. In particular, 
no approach for attacking integrality questions for
multi-variable mirror maps has been available so far. 

In this paper, we present multi-variate versions for all of
(D2)--(D5); all of them seem to be new. 
Our multi-variate extension of (D2) is the content of 
Lemma~\ref{lem:dwork2} in Section~\ref{sec:outline}, 
our multi-variate version of (D3) can be found
in Lemma~\ref{lem:333} in Section~\ref{sec:outline}, 
while Lemma~\ref{lem:Dcomb}  
in Section~\ref{sec:comblemma}
provides our multi-variate extension of (D4). On the other hand, 
we state our
multi-variate extension of item~(D5) in Theorem~\ref{theo:1}
below. Since its one-variable
special case enabled Dwork to address the question of analytic
extension of certain ratios of generalised $p$-adic hypergeometric  
series in one variable, we expect our result below to be
the appropriate tool for analogous studies of multivariable $p$-adic
hypergeometric series.

For the statement of our multi-variate theorem on formal congruences,
we need some standard multi-index notation. Namely, given a positive
integer $d$, a real number $\la$, and vectors 
$\mathbf m=(m_1,m_2,\dots,m_d)$ and
$\mathbf n=(n_1,n_2,\dots,n_d)$ in $\R^d$, we write 
$\bm+\bn$ for $(m_1+n_1,m_2+n_2,\dots,m_d+n_d)$,
$\la\bm$ for $(\la m_1,\la m_2,\dots,\la m_d)$,
we write $\bm\ge\bn$ if and only if
$m_i\ge n_i$ for $i=1,2,\dots,d$, and
we write $\bnull$ for $(0,0,\dots,0)\in \Z^d$ and
$\bid$ for $(1,1,\dots,1)\in\Z^d$.

\begin{theo} \label{theo:1}
Let 
$A:\mathbb{Z}_{\ge 0}^d\to \mathbb{Z}_p\setminus\{0\}$
and 
$g:\mathbb{Z}_{\ge 0}^d\to \mathbb{Z}_p\setminus\{0\}$ be maps satisfying the
following three properties:

\begin{enumerate} 
\item[$(i)$] $v_p\big(A(\bnull)\big) =0$;
\item[$(ii)$] $A(\bn)\in g(\bn)\mathbb{Z}_p$;
\item[$(iii)$] for all non-negative integers $s$ and 
all integer vectors $\bv,\bu,\bn\in\Z^d$ with $\bv,\bu,\bn\ge \bnull$ with 
$0\le v_i<p$ and
$0\le u_i<p^s$, $i=1,2,\dots,d$, 
$$
\frac{A(\bv +p\bu+\bn p^{s+1})}{A(\bv +p\bu)}- 
\frac{A(\bu+\bn p^s)}{A(\bu)}
\in p^{s+1} \,\frac{g(\bn)}{g(\bv +p\bu)}\,\mathbb{Z}_p.
$$
\end{enumerate}

Then, for all non-negative integers $s$ and all 
integer vectors $\bm,\bK,\ba\in\Z^d$ with $\bm\ge\bnull$ and $0\le a_i<p$,
$i=1,2,\dots,d$, 
we have
\begin{equation*}
\sum_{p^s\bm\le\bk\le p^s(\bm+\bid)-\bid}^{} 
\big( A(\ba +p\bk)A(\bK-\bk)-A(\ba +p(\bK-\bk))A(\bk)\big)
\in p^{s+1}g(\bm)\mathbb{Z}_p,
\end{equation*}
where we extend $A$ to $\mathbb{Z}^d$ by $A(\bn)=0$ if there is an $i$
such that $n_i<0$.
\end{theo}

While the proofs of Lemmas~\ref{lem:dwork2} and \ref{lem:Dcomb}
(corresponding to items (D2) and (D4)) are relatively straightforward
extensions of the one-variable proofs given in 
\cite[Ch.~14, p.~76]{lang} and \cite[proof of Lemma~4.2]{dwork},
respectively, the proofs of Lemma~\ref{lem:333} and
Theorem~\ref{theo:1} (corresponding to items (D3) and (D5)) 
need new ideas.
The proof of Lemmas~\ref{lem:dwork2} is given in
Section~\ref{sec:dwork2}. Section~\ref{sec:333} is devoted to the
proof of Lemma~\ref{lem:333}. Even in the one-dimensional case, this
proof is new, as it simplifies the earlier proofs 
\cite[proofs of Lemma~1, respectively Lemma~5]{kr} and
\cite[proof of Lemma~3]{kr2}. In fact, it turned out, that these
earlier proofs could not be extended to the multi-variate case.
The proof of Lemma~\ref{lem:Dcomb} can be found in
Section~\ref{sec:comblemma}. Finally, in
Section~\ref{sec:congruence} we prove Theorem~\ref{theo:1}.

The main application of our multi-variate theory of formal congruences
that we present in this paper concerns the proof that, for a large class of
multi-variable mirror maps, their Taylor coefficients are
integers. We state the corresponding general theorem in the next
subsection. The subsequent subsection collects some particularly
interesting special cases and consequences.

\subsection{A family of GKZ functions and their associated mirror maps}

In order to state the results in this section conveniently, we need to
further enlarge our set of multi-index notations given before
Theorem~\ref{theo:1}. Given vectors 
$\mathbf m=(m_1,m_2,\dots,m_d)$ and
$\mathbf n=(n_1,n_2,\dots,n_d)$ in $\R^d$, we write 
$\bm\cdot\bn$ for the scalar product $m_1n_1+m_2n_2+\dots+m_dn_d$,
and we write $\v{\bm}$ for 
$m_1+m_2+\dots+m_d$. Furthermore, given a vector
$\bz=(z_1,z_2,\dots,z_d)$ of variables and
$\bn=(n_1,n_2,\dots,n_d)\in\Z^d$, we write
$\bz^{\bn}$ for the {\em product} $z_1^{n_1}z_2^{n_2}\cdots z_d^{n_d}$.
On the other hand, if $n$ is an integer, we write
$\bz^n$ for the {\it vector} $(z_1^n,z_2^n,\dots,z_d^n)$.

\medskip

Given $k$ vectors $\bN^{(j)}=(N^{(j)}_1,N^{(j)}_2,\dots,N^{(j)}_d)\in\Z^d$, $j=1, \ldots, k$, 
with $\bN^{(j)}\ge \bnull$, let us define the series 
$$
F_\bN(\bz) = \sum_{\bm\ge\bnull}^{}\bz^\bm
\prod _{j=1} ^{k} \frac{(\bN^{(j)}\cdot \bm)!}
{\prod _{i=1} ^{d}m_i!^{N_i^{(j)}}}
=\sum_{\bm\ge\bnull}^{}\bz^\bm
\prod _{j=1} ^{k} \frac{\big(\sum _{i=1} ^{d}N^{(j)}_im_i\big)!}
{\prod _{i=1} ^{d}m_i!^{N_i^{(j)}}}.
$$ 
Since the Taylor coefficients of $F_\bN(\bz)$ are products of 
multinomial coefficients, it follows that $F_\bN(\bz) \in
1+\sum_{i=1}^d z_i\mathbb{Z}[[\bz]]$, 
where $\mathbb{Z}[[\bz]]$ denotes the set of all (formal) power series 
in the variables $z_1,z_2,\dots,z_d$ with integer coefficients.

This series is  a 
GKZ hypergeometric function~(\footnote{See~\cite{stienstra} for 
an introduction to these functions, which are 
a far-reaching generalisation 
of the classical hypergeometric functions to several variables.}) 
and it is known to ``come from geometry,''
i.e., it can be viewed as the period of certain 
multi-parameter families of algebraic varieties in a product 
of weighted projective spaces (see~\cite{hosono} for details). 
It satisfies a linear differential system
$\{\mathcal{L}_{i,\bN}(F_\bN)=0: i=1, \ldots, d\}$ defined by 
the operators
$$
\mathcal{L}_{i,\bN}=\theta_i^{N_i^{(1)}+\cdots + N_i^{(k)}} - 
z_i\prod_{j=1}^k\prod_{r_j=1}^{N_i^{(j)}}\Big(\sum_{\ell=1}^d N_\ell^{(j)}\theta_\ell + r_j\Big), 
\quad i=1, \ldots, d,
$$
where $\theta_i=z_i\frac{\dd}{\dd z_i}$. 
Amongst the other solutions of this system, we find  the $d$ functions 
$\log(z_i)F_\bN(\bz)+ G_{i,\bN}(\bz)$, $i=1, \ldots, d$, defined by  
$$
 G_{i,\bN}(\bz) =  \sum_{\bm\ge\bnull}^{}\bz^\bm
\bigg( 
\sum_{j=1}^k N_i^{(j)}H_{\bN^{(j)}\cdot \bm} - H_{m_i} \sum_{j=1}^k N_i^{(j)}
\bigg)
\prod _{j=1} ^{k} \frac{(\bN^{(j)}\cdot \bm)!}
{\prod _{i=1} ^{d}m_i!^{N_i^{(j)}}}.
$$ 
Here and in the rest of the article, 
$H_m = \sum_{j=1}^m 1/j$
denotes the $m$-th harmonic number, with the convention $H_0=0$.

This set of solutions enables us to define $d$
{\it canonical coordinates} $q_{i,\bf N}(\bz)$ by 
$$
q_{i,\bf N}(\bz) = z_i\exp\big(G_{i,\bN}(\bz)/F_\bN(\bz)\big),
$$
which are objects with many fundamental properties for the ``mirror symmetry'' study of 
the underlying 
multi-parameter families of varieties. 
The compositional inverse of the map
$$
\mathbf z\mapsto 
(q_{1,\mathbf N}(\mathbf z),q_{2,\mathbf
  N}(\mathbf z), \dots,q_{d,\mathbf N}(\mathbf z))
$$
defines the vector $
(z_{1,\mathbf N}(\mathbf q),z_{2,\mathbf
  N}(\mathbf q), \dots,z_{d,\mathbf N}(\mathbf q))$ of {\it mirror
  maps}. In this paper, 
by abuse of terminology, we will 
also use the term ``mirror map'' for any  
canonical coordinate.~(\footnote{Canonical coordinates and mirror
maps have distinct  
geometric meanings. However, 
in the number-theoretic study undertaken in the present paper,
they play strictly the same role, because 
$q_{i,\mathbf N}(\mathbf z)\in
z_i\mathbb{Z}[[\mathbf z]]$, $i=1,2,\dots,d$, implies that 
$z_{i,\mathbf N}(\mathbf q)\in 
q_i\mathbb{Z}[[\mathbf q]]$, $i=1,2,\dots,d$, and
conversely.})  

Let us define the series  
$$
G_{\bL,\bN}(\bz) = \sum_{\bm\ge\bnull}^{}\bz^\bm
H_{\bL\cdot \bm}
\prod _{j=1} ^{k} \frac{\left(
\bN^{(j)}\cdot\bm\right)!}{
\prod _{i=1} ^{d}m_i!^{N_i^{(j)}}} \in \sum_{i=1}^d z_i\mathbb{Q}[[\bz]],
$$
where $\bL\in\mathbb{Z}^d$ is $\ge \bnull$.
For any $i=1, \ldots, d$, the function $G_{i,\bN}(\bz)$ is a 
finite linear combination with integer coefficients in the functions 
$G_{\bL,\bN}(\bz)$, where the summation runs over various vectors~$\bf L$, 
each one with the property that  
$\bnull \le \bL \le \bN^{(j(\bL))}$ for some $j(\bL)\in \{1, \ldots, k\}$.  
Therefore, the following theorem concerns as well
our mirror maps $q_{i,\bf N}(\bz)\in z_i\mathbb{Z}[[\bz]]$ 
(see Corollary~\ref{cor:mirror}).

\begin{theo} \label{conj:1}
Let $d$ and $k$ be positive integers.
For all vectors $\bL=(L_1,L_2,\dots,L_d)\in\Z^d$ and
$\bN^{(j)}=(N^{(j)}_1,N^{(j)}_2,\dots,N^{(j)}_d)\in\Z^d$, $j=1,2,\dots,k$,
with $\bnull \le \bL\le \bN^{(1)}$, $\bN^{(2)}\ge \bnull, \ldots, \bN^{(k)}\ge \bnull$, we have
$$
q_{\bL,\bN}(\bz):=\exp\big(G_{\bL,\bN}(\bz)/F_\bN(\bz)\big) \in 
\mathbb{Z}[[\bz]].
$$
\end{theo}

\begin{Remarks}\label{rem:1}
(a) Given the fact that the canonical coordinates $q_{i,\mathbf
N}(\mathbf z)$ (which, in their turn, define the mirror maps 
$z_{i,\bN}(\mathbf q)$) can be expressed as products of several series of the
form $q_{\bL,\bN}(\bz)$ (with varying $\bL$), 
we call $q_{\bL,\bN}(\bz)$ a {\em mirror-type map}.

(b) By carefully going through our arguments, one sees that minor
modifications lead to the slightly stronger statement that,
under the assumptions of Theorem~\ref{conj:1}, we have
$$
\exp\big(G_{\bL,\bN}(\bz)/F_\bN(\bz)\big) \in 
\prod _{j=2} ^{k}\big(\min\{N_1^{(j)},N_2^{(j)},\dots,
N_d^{(j)}\}\big)!\,\mathbb{Z}[[\bz]].
$$
\end{Remarks}

The statement of the theorem 
might suggest that $\bN^{(1)}$ plays a special role amongst the vectors 
$\bN^{(1)}, \bN^{(2)}, \ldots, 
\bN^{(k)}$. Of course, this is not the case: by symmetry, given any 
$j\in\{1, \ldots, k\}$, a similar result 
holds for any $\bL$ such that $\bnull \le \bL \le \bN^{(j)}$. This 
remark implies the following result for the mirror maps
$q_{i,\bf N}(\bz)\in z_i\mathbb{Z}[[\bz]]$, 
proving a conjecture of Batyrev and van Straten
\cite[Conjecture~7.3.4]{batstrat} for
a large family of canonical coordinates in several variables.

\begin{coro} \label{cor:mirror}
Let $d,k$ be positive integers.
For all vectors 
$\bN^{(j)}=(N^{(j)}_1,N^{(j)}_2,\dots,N^{(j)}_d)\in\Z^d$, $j=1,2,\dots,k$,
with $\bN^{(1)}\ge\bnull$, $\bN^{(2)}\ge \bnull, \ldots, \bN^{(k)}\ge
\bnull$, we have 
$
q_{i,\bN}(\bz)\in
\mathbb{Z}[[\bz]]$, $i=1,2,\dots,d$.
\end{coro}

We outline the proof of Theorem~\ref{conj:1} in
Section~\ref{sec:outline}, thereby showing how the various pieces of
our multi-variate theory of formal congruences fit together in order
to prove integrality assertions for multi-variable mirror(-type)
maps. The details are deferred to
Sections~\ref{sec:dwork2}--\ref{sec:lem11}. 

\subsection{Consequences of Theorem~\ref{conj:1}} \label{ssec:specialisation}

In order to illustrate the range of applicability of
Theorem~\ref{conj:1}, we collect in this subsection some examples and
applications that are of particular interest to
multi-variable {\it and\/} one-variable mirror-type maps.

\medskip
(1) A classical multi-variate example, 
studied in detail in~\cite[Sec.~7]{batstrat} and~\cite[Sec.~8.4]{stienstra}, is 
the case of the two parameters ($w$ and $z$ say) 
family of 
hypersurfaces $V$ of degree 
$(3,3)$ in $\mathbb{P}^2(\mathbb{C})\times \mathbb{P}^2(\mathbb{C})$, which is a family of 
Calabi--Yau threefolds. The periods of the associated mirror family of 
Calabi--Yau hypersurfaces can be expressed in term of the double series 
\begin{equation}\label{eq:Fexampleclassique}
F(w,z) = \sum_{m\ge 0}\sum_{n\ge 0} \frac{(3m+3n)!}{m!^3n!^3} \,w^mz^n,
\end{equation}
which is symmetric and holomorphic in $\{(w,z)\in \mathbb{C}^2 : 
\vert w\vert^{1/3}+ \vert z\vert^{1/3}<\frac {1} {3}\}$. 
It is a solution of the linear differential system 
$\{\mathcal{L}_1(F)=0, \mathcal{L}_2(F)=0\}$ defined by the operators
\begin{equation*}
\begin{cases}
\mathcal{L}_1 = \theta_1^3-w(3\theta_1+3\theta_2+1)(3\theta_1+3\theta_2+2)(3\theta_1+3\theta_2+3)
,
\\
\mathcal{L}_2 = \theta_2^3-z(3\theta_1+3\theta_2+1)(3\theta_1+3\theta_2+2)(3\theta_1+3\theta_2+3)
,
\end{cases}
\end{equation*}
where $\theta_1=w\frac{\dd}{\dd w}$ and $\theta_2=z\frac{\dd}{\dd z}$. 

Two solutions of this system are of the form 
$\log(w)F(w,z) + G_1(w,z)$ and $\log(z)F(w,z) + G_2(w,z)$ where $G_1(w,z)$ and $G_2(w,z)$ are holomorphic 
in $\{(w,z)\in \mathbb{C}^2 : 
\vert w\vert^{1/3}+ \vert z\vert^{1/3}<\frac {1} {3}\}$, and 
are given explicitly by
\begin{align*}
G_1(w,z)  &= \sum_{m\ge 0}\sum_{n\ge 0} \big(3H_{3m+3n} - 3H_m \big)\frac{(3m+3n)!}{m!^3n!^3} \,w^mz^n, 
\\
G_2(w,z)  &= \sum_{m\ge 0}\sum_{n\ge 0} \big(3H_{3m+3n} - 3H_n \big)\frac{(3m+3n)!}{m!^3n!^3} \,w^mz^n.
\end{align*}

Let us now define the 
two variable mirror maps $q_1(w,z)=w\exp\big(G_1(w,z)/F(w,z)\big)$ and 
$q_2(w,z)=z\exp\big(G_2(w,z)/F(w,z)\big)$. Here, $q_1(w,z)=q_2(z,w)$, 
but this is not 
the case in general. It was observed in the early 
developments of mirror symmetry theory that  
$q_1(w,z)$ and $q_2(w,z)$ seem to have integral Taylor
coefficients (see the end of Section~7.1 in~\cite{batstrat} for
example).
Corollary~\ref{cor:mirror} with $d=2$, $k=1$, $\mathbf N^{(1)}=(3,3)$
now provides a proof for this observation. 

\medskip
(2) 
Interesting consequences result also by considering the series expansion
$q_{\bL,\bN}(\bz)$ for cases where some or all of 
the variables $z_i$ are equal to each other. The obtained series
is obviously still a formal power series. Furthermore, since the 
initial power series has integer coefficients, any such
specialisation leads again to a series with integer coefficients.
In this way, we can construct many new mirror-type maps, 
and, for several of them, 
this leads to proofs of conjectures in the literature on the
integrality of their Taylor coefficients.

Here, we provide details for a corresponding example derived from the
mirror-type map of item (1). Subsequently, Item~(3) will address
another family of one-variable examples derived from
two-variable series, which, for example, includes the series whose
coefficients form the famous sequences that appear in Ap\'ery's proof of
the irrationality of $\zeta(2)$ and $\zeta(3)$. 
Finally, in Item~(4), we mention briefly  
certain cases studied in~\cite{almkvist, enck, batstrat}.

We  
put $w=z$ in the example~\eqref{eq:Fexampleclassique} 
considered in Item~(1) above and get
$$
f(z)= \sum_{m\ge 0}\sum_{n\ge 0} z^{m+n} \frac{(3m+3n)!}{m!^3n!^3} = \sum_{k=0}^{\infty} 
z^k \frac{(3k)!}{k!^3}\sum_{j=0}^k \binom{k}{j}^3
$$
after rearrangement. This 
map is studied  
in~\cite[Sec.~7.3]{batstrat}, where it is shown to 
be of significance in the theory of mirror symmetry. 
The function $f$ satisfies a Fuchsian differential equation of order $4$ with maximal unipotent 
monodromy at the origin: it is annihilated by 
the minimal operator
$$
\theta^4-3z(7\theta^2+7\theta+2)(3\theta+1)(3\theta+2)-72z^2(3\theta+5)(3\theta+4)(3\theta+2)(3\theta+1).
$$
Another solution 
is $g(z)+\log(z)f(z)$,
where $g(z)$ is given by 
$$
g(z) = \sum_{k=0}^{\infty} 
z^k \frac{(3k)!}{k!^3}\sum_{j=0}^k \binom{k}{j}^3 (3H_{3k}-3H_{k-j}).
$$ 
The function $g(z)$ is a linear combination with integer coefficients 
of the functions
$$
g_{\bL}(z) = \sum_{k=0}^{\infty} 
z^k \frac{(3k)!}{k!^3}\sum_{j=0}^k \binom{k}{j}^3 H_{L_1j+L_2(k-j)},
$$
where $\bL=(L_1, L_2)\in\mathbb{Z}^2$ is such that $0\le L_1, L_2\le 3$. For these $\bL$, 
equating the variables in Theorem~\ref{conj:1} leads to
$$
\exp\big(g_{\bL}(z)/f(z) \big) \in \mathbb{Z}[[z]],
$$
which, in particular, implies the new result that $z\exp\big(g(z)/f(z)\big)\in z\mathbb{Z}[[z]]$.

\medskip

(3) For any integers $\alpha, \beta$ such that $0 \le \beta\le \alpha$, 
we consider the function
\begin{equation}\label{eq:functapery}
\sum_{m\ge 0}\sum_{n\ge 0} \left(\frac{(m+n)!}{m!\,n!}\right)^{\alpha-\beta}
\left(\frac{(2m+n)!}{m!^2n!}\right)^\beta w^mz^n.
\end{equation}
The 
specialisation $w=z$
produces the function 
$$
\mathcal{A}_{\alpha, \beta}(z) = \sum_{k=0}^{\infty} \bigg( \sum_{j=0}^k \binom{k}{j}^\alpha
\binom{k+j}{j}^\beta\bigg) z^k,
$$
to which we associate the function 
$\mathcal{B}_{\al,\beta}(z)+\log(z)\mathcal{A}_{\alpha, \beta}(z)$ defined by
$$
\mathcal{B}_{\al,\beta}(z) = \sum_{k=0}^{\infty} \bigg( \sum_{j=0}^k \binom{k}{j}^\alpha
\binom{k+j}{j}^\beta\big((\al-\beta) H_k -\al H_{k-j}+\beta H_{k+j}\big)\bigg) z^k.
$$

Let $\mathcal{L}_{\al,\be}$ denote the minimal Fuchsian differential
operator that annihilates
$\mathcal{A}_{\alpha, \beta}(z)$: it does not
always have maximal unipotent monodromy at $z=0$, as 
the case $(\al,\be)=(6,0)$ shows (cf.\ 
\cite[Sec.~10]{almkvist}).
The operator $\mathcal{L}_{\al,\be}$ also annihilates
$\mathcal{B}_{\al,\beta}(z)+\log(z)\mathcal{A}_{\alpha, \beta}(z)$
and we define the mirror map
$z\exp\big(\mathcal{B}_{\al,\beta}(z)/\mathcal{A}_{\al,\beta}(z)\big)$. We 
observe that $\mathcal{B}_{\al,\beta}(z)$ is a linear combination with integer 
coefficients in the functions 
$$
\mathcal{B}_{\bL, \al,\beta}(z) =  \sum_{k=0}^{\infty} \bigg( \sum_{j=0}^k \binom{k}{j}^\alpha
\binom{k+j}{j}^\beta  H_{L_1j+L_2(k-j)}\bigg) z^k.
$$
Here, $\bL=(L_1,L_2)\in \mathbb{Z}^2$ is such that $0\le L_1\le 2$ and $0\le L_2\le 1$. For these $\bL$, 
equating the variables in Theorem~\ref{conj:1} leads to
$$
\exp\big(\mathcal{B}_{\bL, \al,\beta}(z)/\mathcal{A}_{\al,\beta}(z)\big) \in
\mathbb{Z}[[z]],
$$
and this implies that  $z\exp\big(\mathcal{B}_{\al,\beta}(z)/\mathcal{A}_{\al,\beta}(z)\big) 
\in z\mathbb{Z}[[z]]$. This example is particularly 
interesting because it proves that maximal unipotent monodromy at the 
origin is not a necessary condition to obtain mirror-type maps
with integer Taylor coefficients.

It is interesting to note that
the Taylor coefficients of $\mathcal{A}_{2,1}(z)$ and
$\mathcal{A}_{2,2}(z)$ form the  
sequences  appearing in  Ap\'ery's proof of the irrationality of 
$\zeta(2)$ and $\zeta(3)$, respectively. Beukers~\cite{beukers} showed
that $\mathcal{A}_{2,1}(z)$ and $\mathcal{A}_{2,2}(z)$ are strongly 
related to modular forms, a fact which also explains 
the integrality properties of the associated mirror-type
maps. (For $p$-adic properties of $\mathcal{A}_{2,1}(z)$, we
refer the reader to \cite{beukers2}.)

\medskip

(4) 
Equating variables in Theorem~\ref{conj:1} can explain the integrality properties of many of 
the mirror-type maps in~\cite{almkvist}, many of which have been incorporated 
in the table~\cite{enck} of ``Calabi--Yau differential equations''.
This table contains a list of more than 300 
Fuchsian differential equations of order $4$ with certain analytic properties, amongst which 
are maximal unipotent monodromy at the origin and conjectural integrality of the instanton-type numbers. Only the first 29 items 
are currently known to have a geometric origin, 
meaning that they 
have an interpretation in mirror symmetry; for example, the 
instanton-type numbers 
in these cases are really 
instanton numbers.
In particular, the table 
contains the 
mirror-type maps of geometric origin considered in 
Sections~8.1, 8.2, 8.3 and 8.4 
of~\cite{batstrat}, which all come from 
equating variables in series covered by Theorem~\ref{conj:1}. 

Although this is not mentioned 
explicitly in~\cite{enck}, it is plausible that the mirror-type maps associated to each example 
of the table have integer Taylor coefficients.
In this direction, we have checked that the functions whose Taylor coefficients are given in 
items $15$ to $23$, 25, 34, 39, 45, 58, 60, 72, 76, 78, 79, 81, 91, 93, 96, 97, 127, 
130, 188, 190 and $191$, are 
specialisations of multi-variable series that can be treated with 
Theorem~\ref{conj:1}. Hence the mirror-type maps associated to these items have integer Taylor 
coefficients. Incidentally, items $1$ to $14$ are all covered by  
the results in~\cite{kr, lianyau, zud}
and therefore, amongst the ``geometric'' items $1$ to $29$, there remains to 
understand only items $24, 26, 27, 28, 29$.

\medskip

We could use many other 
ways of specialisation in 
conjunction with Theorem~\ref{conj:1}, for example 
``weighted'' equating such as 
$z_1=Mz_2^N$ for some integer parameters $M\neq 0$ and $N\ge 1$. 

\section{Outline of the proof of Theorem~\ref{conj:1}}\label{sec:outline}

In this section, we present a decomposition of the proof of 
Theorem~\ref{conj:1} 
into various assertions, which form our multi-variate theory of formal
congruences described in Subsection~\ref{ssec:0}. The individual
assertions will be proved in the later sections. 

The starting point (listed as (D1) in Subsection~\ref{ssec:0})
is the observation that, given a power series $S(\mathbf
z)=S(z_1,z_2,\dots,z_d)$ in $\mathbb Q[[\mathbf z]]$, 
the series $S(\mathbf z)$ is an element of $\mathbb Z[[\mathbf z]]$ if
and only if, for all primes $p$, it is an element of
$\mathbb Z_p[[\mathbf z]]$.

Next, we want to get rid of the exponential 
function in  the definition of the mirror-type map $q_{\bL,\bN}(\bz)$. 
To achieve this, we use a generalisation 
of a lemma attributed to Dieudonn\'e and Dwork in \cite[Ch.~14, p.~76]{lang}
to several variables,
the latter being the univariate case of the following lemma
(corresponding to (D2) in Subsection~\ref{ssec:0}). 

\begin{lem}
\label{lem:dwork2}
For $S(\bz)\in 1+
\sum _{i=1} ^{d}z_i\mathbb{Q}_p[[\bz]]$, we have
$$
S(\bz)\in 1+\sum _{i=1} ^{d}z_i\mathbb{Z}_p[[\bz]]
\quad \text {if and only if}\quad 
\frac{S(\bz^p)}{S(\bz)^p} \in 1+ p\sum _{i=1} ^{d}z_i\mathbb{Z}_p[[\bz]].
$$
\end{lem}

This lemma enables us to prove the following reduction of our problem.

\begin{lem} \label{lem:equivalencesansexp}
Given two formal  series 
$F(\bz)\in 1+\sum _{i=1} ^{d}z_i\mathbb{Z}[[\bz]]$ and 
$G(\bz)\in \sum _{i=1} ^{d}z_i\mathbb{Q}[[\bz]]$, let
$q(\bz):=\exp(G(\bz)/F(\bz))$. Then we have 
$
q(\bz) \in 1+\sum _{i=1} ^{d}z_i\mathbb{Z}_p[[\bz]]$
if and only if
$$
F(\bz)G(\bz^p)-pF(\bz^p)G(\bz) 
\in p\sum _{i=1} ^{d}z_i\mathbb{Z}_p[[\bz]].
$$
\end{lem}
These two lemmas are proved in Sections~\ref{sec:dwork2} 
and \ref{sec:equivalencesansexp}, respectively.

We write $B_\bN(\bm)=\prod _{j=1} ^{k}B(\bN^{(j)},\bm)$, where
\begin{equation} \label{eq:Bj}
B(\bP,\bm)=\frac {\left(
\sum _{i=1} ^{d}P_im_i\right)!} 
{
\prod _{i=1} ^{d}m_i!^{P_i}}
\end{equation}
for all vectors $\bP,\bm\in\Z^d$ with $\bP\ge\bnull$ and $\bm\ge\bnull$,
while we define $B(\bP,\bm)=0$ for vectors $\bm$ for which 
$m_i<0$ for some $i$.
(If we interpret factorials $n!$ as $\Ga(n+1)$, where $\Ga$ stands
for the gamma function, then
this convention is in accordance with the behaviour of the gamma function.)
Note that, using this notation, we have
$F_\bN(\bz)=\sum_{\bm\ge\bnull}^{}\bz^\bm
B_\bN(\bm)$ and $G_{\bL,\bN}(\bz)=\sum_{\bm\ge\bnull}^{}\bz^\bm
H_{\bL\cdot \bm}B_\bN(\bm)$.

As already mentioned, we have $F_\bN(z)\in 1+\sum _{i=1} ^{d}z_i\mathbb{Z}[[\bz]]$ and thus 
we can use Lemma~\ref{lem:equivalencesansexp} 
with $F(\bz)=F_\bN(\bz)$ and $G(\bz)=G_{\bL,\bN}(\bz)$.
The coefficient of $\bz^{\ba+p\bK}$ (with $0 \le a_i<p$ for all $i$)
in the Taylor expansion of the formal power series $F_\bN(\bz)G_{\bL,\bN}(\bz^p)-pF_\bN(\bz^p)G_{\bL,\bN}(\bz)$
can be written in the form
\begin{equation*}
C(\ba+p\bK) = \sum_{\bnull\le \bk\le \bK} B_\bN(\ba+p\bk)B_\bN(\bK-\bk) 
\Big(H_{
\bL\cdot (\bK-\bk)}- pH_{
(\bL\cdot \ba+p\bL\cdot \bk)}\Big).
\end{equation*}
Lemma~\ref{lem:equivalencesansexp} tells us that we have to show that $C(\ba+p\bK)$ is in $p\mathbb{Z}_p.$

To prove this, we will proceed step by step. 
First, because of the congruence~(\footnote{%
This is an immediate consequence 
of the  
identity $\displaystyle H_J=\sum_{j=1}^{\lfloor J/p\rfloor} \frac1{pj}+ \sum_{j=1, p\nmid j}^J \frac1j$.}) 
$$pH_{(\bL\cdot \ba+p\bL\cdot \bk)} \equiv 
H_{\fl{\frac {1} {p}\bL\cdot \ba}+\bL\cdot \bk} 
\mod p\mathbb{Z}_p,
$$ 
we obtain
\begin{equation*}
C(\ba+p\bK) \equiv \sum_{\bnull\le \bk\le \bK} B_\bN(\ba+p\bk)B_\bN(\bK-\bk) 
\Big(H_{
\bL\cdot(\bK-\bk)}- H_{
\fl{\frac {1} {p}\bL\cdot \ba}+\bL\cdot \bk}\Big) 
\mod p \mathbb{Z}_p.
\end{equation*}

Then, the following lemma (corresponding to (D3) in
Subsection~\ref{ssec:0}) is proved in Section~\ref{sec:333}.

\begin{lem} \label{lem:333}
For any prime $p$, vectors 
$\ba,\bk,\bL,\bN^{(1)}\in\Z^d$ with $\bk\ge\bnull$,
$\bnull\le\bL\le \bN^{(1)}$, and 
$0\le a_i<p$ for $i=1,2,\dots,d$, we have 
$$
B(\bN^{(1)},\ba+p\bk) \Big( H_{
\fl{\frac {1} {p}\bL\cdot \ba}+\bL\cdot \bk}-
 H_{\bL\cdot \bk} \Big) \in p\mathbb{Z}_p,
$$
where $B(\bN^{(1)},\ba+p\bk)$ is defined in \eqref{eq:Bj}.
\end{lem}
Since $B(\bN^{(1)},\ba+p\bk)$ is a factor of $B_\bN(\ba+p\bk)$,
it follows that
\begin{equation*}
C(\ba+p\bK) \equiv \sum_{\bnull\le \bk\le \bK} B_\bN(\ba+p\bk)B_\bN(\bK-\bk)
\big(H_{
\bL\cdot(\bK-\bk)}-H_{\bL\cdot \bk}\big) \mod p \mathbb{Z}_p.
\end{equation*}
For the right-hand side, we obviously have
\begin{multline}
\sum_{\bnull\le \bk\le \bK} B_\bN(\ba+p\bk)B_\bN(\bK-\bk)
\big(H_{
\bL\cdot(\bK-\bk)}-H_{\bL\cdot \bk}\big)
\\
= - \sum_{\bnull\le \bk\le \bK} H_{\bL\cdot \bk}
\big( B_\bN(\ba+p\bk)B_\bN(\bK-\bk)-B_\bN(\ba+p(\bK-\bk))B_\bN(\bk)\big).
\label{eq:obvious}
\end{multline}

We now use the multi-variable extension of the combinatorial lemma of Dwork 
(corresponding to (D4) in Subsection~\ref{ssec:0}; 
stated here as Lemma~\ref{lem:Dcomb} in Section~\ref{sec:comblemma},
with proof in the same section) in order
to decompose the sum over $\bk$.
Namely, if in Lemma~\ref{lem:Dcomb} we let $Z(\bk)=H_{\bL\cdot \bk}$,
$$W(\bk)=B_\bN(\ba+p\bk)B_\bN(\bK-\bk)-B_\bN(\ba+p(\bK-\bk))B_\bN(\bk),$$
and choose an integer $r$ that satisfies $p^{r-1}>\max\{K_1,K_2,\dots,K_d\}$,
then
\begin{multline} 
C(\ba+p\bK)
\equiv
-\sum_{s=0}^{r-1}\sum _{\bnull\le\bm\le (p^{r-s}-1)\bid} ^{}
\big(H_{
\sum _{i=1} ^{d}L_im_ip^s}-H_{
\sum _{i=1} ^{d}L_i\fl{\frac {m_i}p}p^{s+1}}\big)
\\
\cdot\sum _{p^s\bm\le\bk\le p^s(\bm+\bid)-\bid} ^{}
\big( B_\bN(\ba+p\bk)B_\bN(\bK-\bk)-B_\bN(\ba+p(\bK-\bk))B_\bN(\bk)\big)\\
\mod p \mathbb{Z}_p.
\label{eq:obvious2}
\end{multline}
(Since
for the first term appearing on the right-hand side of \eqref{eq:comb} we have
$Z(\bnull)\overline W_r(\bnull)=H_0\overline W_r(\bnull)=0$, 
the right-hand sides 
of~\eqref{eq:obvious} and~\eqref{eq:obvious2} are in fact equal.)

To deal with the sum over $\bk$ in~\eqref{eq:obvious2}, we invoke
Theorem~\ref{theo:1} 
(corresponding to (D5) in Subsection~\ref{ssec:0}).
(Its proof is given in Section~\ref{sec:congruence}).
We
show in Section~\ref{sec:reduction} that 
Theorem~\ref{theo:1} can be applied with $A=g=B_\bN$. 
Using this, we obtain
\begin{equation}\label{eq:consequencecongruence}
\sum_{p^s\bm\le\bk\le p^s(\bm+\bid)-\bid}^{} 
\big(
B_\bN(\ba+p\bk)B_\bN(\bK-\bk)-B_\bN(\ba+p(\bK-\bk))B_\bN(\bk)\big)
\in p^{s+1}B_\bN(\bm)\mathbb{Z}_p.
\end{equation}

We now have to deal with the harmonic sums 
$$
H_{\sum _{i=1} ^{d}L_im_ip^s}-H_{
\sum _{i=1} ^{d}L_i\fl{\frac {m_i}p}p^{s+1}}
$$
occurring 
on the right-hand side of~\eqref{eq:obvious2}. 
In this regard, we prove the following lemma in Section~\ref{sec:lem11}.
(As we show there, it can be reduced to Lemma~\ref{lem:333}.)

\begin{lem} \label{lem:11}
For all primes $p$, vectors
$\bm,\bL,\bN^{(1)},\bN^{(2)},\dots,\bN^{(d)}\in\Z^d$ with
$\bm,\bL,\bN^{(1)},\bN^{(2)},\break
\dots,\bN^{(d)}\ge\bnull$, we have
\begin{equation}\label{eq:congruenceharmonic}
B_\bN(\bm)\big(H_{
\sum _{i=1} ^{d}L_im_ip^s}-H_{
\sum _{i=1} ^{d}L_i\fl{\frac {m_i}p}p^{s+1}}\big) \in \frac{1}{p^s}\,\mathbb{Z}_p.
\end{equation}
\end{lem}

Consequently, putting the congruences~\eqref{eq:consequencecongruence} 
and~\eqref{eq:congruenceharmonic} together, 
it follows from~\eqref{eq:obvious2} that $C(\ba +p\bk)$ is congruent mod $p\mathbb{Z}_p$ 
to a multiple sum (over $s$ and $\bm$) whose terms  
are all in $p\mathbb{Z}_p$. Hence, we have 
established that
$$
C(\ba +p\bk)\in p\mathbb{Z}_p.
$$
This concludes our outline of the proof Theorem~\ref{conj:1}. 


\section{Proof of Lemma \ref{lem:dwork2}}
\label{sec:dwork2}

{\sc Proof of the ``only if" part.} 
We have to show that if $S(\bz)\in 1 + 
\sum _{i=1} ^{d}z_i\mathbb{Z}_p[[\bz]]$,
then
$$
\frac{S(\bz^p)}{S(\bz)^p} \in 1+p\sum _{i=1} ^{d}z_i\mathbb{Z}_p[[\bz]].
$$
To do this, we set $S(\bz)=\sum_{\bi\ge \bnull} a_{\bi}\bz^\bi$. 
The congruence 
$(u+v)^p\equiv u^p+v^p\mod p\mathbb{Z}_p$ and Fermat's Little Theorem
imply that
\begin{align*}
S(\bz)^p &= \bigg(\sum_{\bi\ge \bnull} a_{\bi}\bz^\bi\bigg)^p \equiv 
\sum_{\bi\ge \bnull} a_{\bi}^p\bz^{p\bi} \mod 
p\sum _{i=1} ^{d}z_i\mathbb{Z}_p[[\bz]]
\\
&\equiv \sum_{\bi\ge \bnull} a_{\bi}\bz^{p\bi} \mod p
\sum _{i=1} ^{d}z_i\mathbb{Z}_p[[\bz]].
\end{align*}
This means that $S(\bz)^p = S(\bz^p)+pH(\bz)$ with $H(\bz)\in 
\sum _{i=1} ^{d}z_i\mathbb{Z}_p[[\bz]]$. 
Hence,
$$
\frac{S(\bz^p)}{S(\bz)^p} = 1 -p\,\frac{H(\bz)}{S(\bz)^p} \in 1+p
\sum _{i=1} ^{d}z_i\mathbb{Z}_p[[\bz]],
$$
because the formal series $S(\bz) \in 1 + 
\sum _{i=1} ^{d}z_i\mathbb{Z}_p[[\bz]]$ is invertible
in $\mathbb{Z}_p[[\bz]]$.

\medskip

{\sc Proof of the ``if" part.}
Suppose that $S(\bz^p)=S(\bz)^pR(\bz)$ with
$R(\bz)=1+ p\sum_{\v\bi\ge \bid}b_{\bi}\bz^\bi \in 1+ 
p\sum _{i=1} ^{d}z_i\mathbb{Z}_p[[\bz]]$ and $S(\bnull)=1$. 
Set
$S(\bz)=\sum_{\bi\ge 0} a_{\bi}\bz^\bi$. We have $a_{\bnull}=1$, and 
we proceed by induction on $\v\bi$ to show that $a_{\bi}\in \mathbb{Z}_p$. 

So, let us assume that $a_{\bi}\in \mathbb{Z}_p$ for all
vectors $\bi\in\Z^d$ with $\v\bi\le r-1$.
Let $\bn\in\Z^d$ be a vector with $\v\bn=r$. The Taylor coefficient $C_{\bn}$ of
$\bz^\bn$
in $S(\bz^p)$ is
$$
\begin{cases}
a_{\frac {1} {p}\bn} \quad \textup{if }p\mid n_1,\ p\mid n_2,\dots,
\ p\mid n_d\,  ;
\\
0 \quad \textup{otherwise.}
\end{cases}
$$
The Taylor coefficient $C_{\bn}$ is at the same time also equal to the coefficient
of $\bz^\bn$ in the expansion of the series
$$
\bigg(\sum_{\bi\ge\bnull}
a_{\bi}\bz^\bi\bigg)^p\bigg(1+ p
\sum_{\bi\ge \bnull,  \, \v\bi\ge1}
b_{\bi}\bz^\bi\bigg).
$$
The coefficient of $\bz^\bn$ in this series is thus 
$
C_{\bn}=B_{\bn} + pD_{\bn},
$
where 
\begin{equation} \label{eq:Bn}
B_{\bn} = \sum_{\bi^{(1)}+\dots+\bi^{(p)}=\bn}
a_{\bi^{(1)}}\cdots a_{\bi^{(p)}} 
\end{equation}
and
\begin{equation} \label{eq:Dn}
D_{\bn} = 
\underset{\v{\bi^{(p+1)}}>0}
{\sum_{\bi^{(1)}+\dots+\bi^{(p+1)}=\bn}}
a_{\bi^{(1)}}\cdots a_{\bi^{(p)}}b_{\bi^{(p+1)}}
\in \mathbb{Z}_p.
\end{equation}

\medskip

{\sc Case 1}. If $p\mid n_1$, \dots, $p \mid n_d$, 
in the multiple sum $B_{\bn}$ a term
$\prod _{\ell=1} ^{m}a_{\bi_\ell}^{e_\ell}$ with $a_{\bi_{\ell_1}}\ne 
a_{\bi_{\ell_2}}$ occurs
\begin{equation} \label{eq:multinom}
\frac {(e_1+\cdots+e_m)!} {e_1!\cdots e_m!}=
\frac {p!} {e_1!\cdots e_m!}
\end{equation}
times.
The multinomial coefficient \eqref{eq:multinom} is an integer
divisible by $p$, except if $m=1$ and $e_1=p$; that is, if we are
looking at the term $a_{\frac {1} {p}\bn}^p$, which occurs with
coefficient $1$ in $B_\bn$.
The term $a_{\bn}$ appears in the form
$pa_{\bn}a_{\bnull}^{p-1}=pa_{\bn}$ in the expression \eqref{eq:Bn}
for $B_\bn$. 
For all other terms in the sum on the right-hand side of
\eqref{eq:Bn}, we have $\v{\bi^{(\ell)}}<\v\bn$ for
$\ell=1,2,\dots,p$. Hence, the induction hypothesis applies to all
the factors in the corresponding terms $a_{\bi^{(1)}}\cdots a_{\bi^{(p)}}$, 
whence $B_{\bn} =  p a_{\bn} + a_{\frac {1} {p}\bn}^p \mod  p\mathbb{Z}_p$.

In the multiple sum \eqref{eq:Dn} for $D_{\bn}$, the condition 
$\v{\bi^{(p+1)}}>0$ guarantees that
$\v{\bi^{(\ell)}}<\v\bn$ for $\ell=1, \ldots, p$, and therefore we can
apply the induction hypothesis to each factor
$a_{\bi^{(\ell)}}$. This shows that $D_{\bn}\in \mathbb{Z}_p.$ 

We therefore have
$$
a_{\frac {1} {p}\bn} = C_{\bn} \equiv p a_{\bn} +  
a_{\frac {1} {p}\bn}^p  \mod p\mathbb{Z}_p,
$$ 
whence,
$$
pa_{\bn} \equiv a_{\frac 1p\bn} - a_{\frac 1p\bn}^p  \mod
p\mathbb{Z}_p.
$$
This shows that $a_{\bn}\in \mathbb{Z}_p$ since $ a_{\frac 1p\bn} - a_{\frac 1p\bn}^p \in p\mathbb{Z}_p$ 
by Fermat's Little Theorem.

\medskip

{\sc Case 2}. If $p\nmid n_i$ for some $i$ between $1$ and $d$, 
the only change compared to the preceding case is that the term
$a_{\frac 1p\bn}^p$ does not occur. Therefore, in this case we have
$$
0=C_{\bn} \equiv p a_{\bn} \mod p \mathbb{Z}_p.
$$
Hence,
$$
pa_{\bn} \equiv   0 \mod p \mathbb{Z}_p,
$$
which shows again that $a_{\bn}\in \mathbb{Z}_p$.

This completes the proof of the lemma.\quad \quad \qed

\section{Proof of Lemma~\ref{lem:equivalencesansexp}}
\label{sec:equivalencesansexp}

We begin by showing that, if 
$S(\bz)\in\sum _{i=1} ^{d}z_i\mathbb{Q}_p[[\bz]]$, then
$$
\exp(S(\bz)) \in 1+ 
\sum _{i=1} ^{d}z_i\mathbb{Z}_p[[\bz]] 
\quad \text {if and only if}\quad 
S(\bz^p)-pS(\bz) \in 
p\sum _{i=1} ^{d}z_i\mathbb{Z}_p[[\bz]].
$$
The formal power series $\exp(X)$ and $\log(1+X)$ are defined by
their usual expansions.

\medskip

{\sc Proof of the ``if" part.}  By Lemma~\ref{lem:dwork2} with $S(\bz)$
replaced by $\exp(S(\bz))$, we have
$$
\exp\big(S(\bz^p)-pS(\bz)\big) \in 1+
p\sum _{i=1} ^{d}z_i\mathbb{Z}_p[[\bz]].
$$
Therefore, we have $S(\bz^p)-pS(\bz)=\log(1+pH(\bz))$ with 
$H(\bz)\in 
\sum _{i=1} ^{d}z_i\mathbb{Z}_p[[\bz]]$. This  yields
$$
S(\bz^p)-pS(\bz) = -\sum_{n=1}^{\infty} \frac{p^n}{n}(-H(\bz))^n \in 
p\sum _{i=1} ^{d}z_i\mathbb{Z}_p[[\bz]]
$$
since $v_p(p^n/n)\ge 1$ for all integers $n\ge 1$.

\medskip

{\sc Proof of the ``only if" part.} We have $S(\bz^p)-pS(\bz)=pJ(\bz)$
with $J(\bz)\in 
\sum _{i=1} ^{d}z_i\mathbb{Z}_p[[\bz]]$. Therefore, we have
$$
\exp\big(S(\bz^p)-pS(\bz)\big) =1+ \sum_{n=1}^{\infty} \frac{p^n}{n!}J(\bz)^n\in 1+ 
p\sum _{i=1} ^{d}z_i\mathbb{Z}_p[[\bz]],
$$
since 
$$
v_p\left(\frac {p^n}{n!}\right)= n-\sum_{k=1}^{\infty}\fl{\frac{n}{p^k}}
> n - \sum_{k=1}^{\infty}\frac{n}{p^k} = \frac{p-2}{p-1}\,n\ge 0.
$$
By Lemma~\ref{lem:dwork2} with $S(\bz)$ replaced by $\exp\big(S(\bz)\big)$, 
it follows that
$$
\exp\big(S(\bz)\big)\in 1+ 
\sum _{i=1} ^{d}z_i\mathbb{Z}_p[[\bz]].
$$

\medskip

In order to finish the proof of the lemma, we observe that for
$S=G/F$ with 
$F(\bz) \in 1+\sum _{i=1} ^{d}z_i\Z_p[[\bz]]$, 
we have the equivalence 
$$
S(\bz^p)-pS(\bz) \in 
p\sum _{i=1} ^{d}z_i\mathbb{Z}_p[[\bz]] 
\quad \text {if and only if}\quad 
F(\bz)G(\bz^p)-pF(\bz^p)G(\bz) \in 
p\sum _{i=1} ^{d}z_i\mathbb{Z}_p[[\bz]],
$$
since $F(\bz)$ is invertible in $\mathbb{Z}_p[[\bz]]$.\quad \quad
\qed

\section{Proof of Lemma \ref{lem:333}} 
\label{sec:333}

The proof below generalises Section~6 of \cite{kr} to higher
dimensions. However, it differs from the former even in the case
$d=1$, and thus provides an alternative argument.

For convenience, we shall drop the upper index in $N^{(1)}_i$ in this
section, that is, we write
$$B(\bN^{(1)},\bm)=\frac {(
\bN^{(1)}\cdot \bm)!} 
{\prod _{i=1} ^{d}m_i!^{N^{(1)}_i}}=
\frac {(\bN\cdot\bm)!} 
{
\prod _{i=1} ^{d}m_i!^{N_i}}.$$

We note that
the $p$-adic valuation of 
$B(\mathbf N,\ba+p\bk)$ is equal to
\begin{equation} \label{eq:111} 
v_p\big(B(\mathbf N,\ba+p\bk)\big)=
\sum _{\ell =1} ^{\infty}\left(\fl{\sum _{i=1} ^{d}\frac {N_i(a_i+pk_i)}
{p^\ell }}- \sum _{i=1} ^{d}N_i\fl{\frac {a_i+pk_i} {p^\ell }}\right).
\end{equation}
By definition of the harmonic numbers, we have
\begin{equation*}
H_{
\fl{\frac {1} {p}\bL\cdot \ba}+\bL\cdot\bk}-
 H_{\bL\cdot\bk}
 =\frac {1} {\bL\cdot\bk+1}+\frac {1} {\bL\cdot\bk+2}+\dots+
\frac {1} {\bL\cdot\bk+\fl{\frac {1} {p}\bL\cdot\ba}}.
\end{equation*}
It therefore suffices to show that
\begin{equation} \label{eq:100} 
v_p\big(B(\mathbf N,\ba+p\bk)\big)\ge 1+ 
\max_{1\le\ep\le \fl{\frac {1} {p}\bL\cdot\ba}}
v_p\bigg(\bL\cdot\bk+\ep\bigg).
\end{equation}

For a given integer 
$\ep$ with $1\le\ep\le \lfloor{\frac {1} {p}\bL\cdot\ba}\rfloor$, 
let $\al=v_p(\bL\cdot\bk+\ep)$. Furthermore, let $\ell$ be an integer with $1\le
\ell\le \al +1$. We write $k_i$ in the form
$$k_i=k_{i,0}+k_{i,1}p+\dots+k_{i,\ell-1}p^{\ell-1},$$
where $0\le k_{i,j}<p$ for $0\le j\le\ell-2$, and $k_{i,\ell-1}\ge0$. 
The reader should note
that this representation of $k_i$ is unique, with no upper bound on
$k_{i,\ell-1}$.
We substitute this in \eqref{eq:111}, to obtain
\begin{align}
\notag
v_p\big(B(\mathbf N,\ba+p\bk)\big)&\ge
\sum _{\ell =1} ^{\al +1}\Bigg(\fl{\sum _{i=1} ^{d}\frac 
{N_i\big(a_i+p(k_{i,0}+k_{i,1}p+\dots+k_{i,\ell-1}p^{\ell-1})\big)}
{p^\ell }}\\
\notag
&\kern3.9cm
- \sum _{i=1} ^{d}N_i\fl{\frac 
{a_i+p(k_{i,0}+k_{i,1}p+\dots+k_{i,\ell-1}p^{\ell-1})} {p^\ell
}}\Bigg)\\
&\ge
\sum _{\ell =1} ^{\al+1}\fl{\sum _{i=1} ^{d}\frac 
{N_i\big(a_i+p(k_{i,0}+k_{i,1}p+\dots+k_{i,\ell-2}p^{\ell-2})\big)}
{p^\ell }}.
\label{eq:3331}
\end{align}
It should be noted, that from the first to the second line the terms
containing $k_{i,\ell-1}$ cancel because they can be put outside of
the floor functions. Subsequently, in the second sum, there remains
the term
$$\fl{\frac 
{a_i+p(k_{i,0}+k_{i,1}p+\dots+k_{i,\ell-2}p^{\ell-2})} {p^\ell }},$$
which vanishes, since $0\le a_i,k_{i,0},k_{i,1},\dots,k_{i,\ell-2}<p$.

On the other hand, we have 
\begin{align} 
\notag
\ep+\sum _{i=1} ^{d}L_ik_i&\le \sum _{i=1} ^{d}L_ik_i+\frac {1}
{p}\sum _{i=1} ^{d}L_ia_i\\ 
\notag
&\le \frac {1} {p} \sum _{i=1} ^{d}L_i(a_i+pk_i)\\
&\le \frac {1} {p} \sum _{i=1} ^{d}L_i(a_i+k_{i,0}p+k_{i,1}p^2+\dots+
k_{i,\ell-2}p^{\ell-1})+
p^{\ell-1}\sum _{i=1} ^{d}L_ik_{i,\ell-1}.
\label{eq:3332}
\end{align}
Since $p^\al \,\big\vert \,\big(\ep+\sum _{i=1} ^{d}L_ik_i\big)$ and $\ell\le \al+1$, 
we have also
$p^{\ell-1}\,\big\vert\, \big(\ep+\sum _{i=1} ^{d}L_ik_i\big)$, which implies that
$$p^{\ell-1}\,\bigg\vert\, \bigg(\ep+\sum _{i=1} ^{d}L_ik_i-p^{\ell-1}\sum _{i=1}
^{d}L_ik_{i,\ell-1}\bigg).$$
On the other hand, we have
\begin{align*}
\ep+\sum _{i=1} ^{d}L_ik_i-p^{\ell-1}\sum _{i=1}
^{d}L_ik_{i,\ell-1}&=\ep+
\sum _{i=1} ^{d}L_i(k_{i,0}+k_{i,1}p+\dots+
k_{i,\ell-2}p^{\ell-2})\\
&\ge \ep>0,
\end{align*}
whence, by \eqref{eq:3332},
\begin{align*}
p^{\ell-1}&\le \ep+\sum _{i=1} ^{d}L_ik_i-p^{\ell-1}\sum _{i=1}
^{d}L_ik_{i,\ell-1}\\
&\le \frac {1} {p} \sum _{i=1} ^{d}L_i(a_i+k_{i,0}p+k_{i,1}p^2+\dots+
k_{i,\ell-2}p^{\ell-1}),\\
\end{align*}
or, equivalently,
$$1\le \frac {1} {p^{\ell}} \sum _{i=1} ^{d}N_i(a_i+k_{i,0}p+k_{i,1}p^2+\dots+
k_{i,\ell-2}p^{\ell-1}).$$
That is to say, the summand of the sum over $\ell$ in \eqref{eq:3331} 
is at least $1$. 
Since $\ell$ was restricted to
$1\le\ell\le \al +1$, this implies that $v_p\big(B(\mathbf
N,\ba+p\bk)\big)\ge \al +1$. The claim \eqref{eq:100} follows
immediately, which finishes the proof of the lemma.\quad \quad \qed

\section{A combinatorial lemma}\label{sec:comblemma}

In this section, we generalise 
a combinatorial lemma due to Dwork
(see~\cite[Lemma 4.2]{dwork}) to several variables. 
\begin{lem} \label{lem:Dcomb}
Let $r$ be a non-negative integer, 
let $Z$ and $W$ be maps from $\Z^d$ to a ring $R$, and let
$$\overline W_r(\bm)=
\sum_{p^r\bm\le\bk\le p^r(\bm+\bid)-\bid}^{} 
W(\bk).
$$
Then
\begin{multline} \label{eq:comb}
\sum _{\bnull\le\bk\le(p^r-1)\bid}
Z(\bk)W(\bk)=
Z(\bnull)\overline W_r(\bnull)\\+
\sum _{s=0} ^{r-1}\Bigg(
\sum _{\bnull\le\bm\le(p^{r-s}-1)\bid}
\left(Z(m_1p^s,\dots,m_dp^s)-Z\left(\fl{\tfrac {m_1}p}p^{s+1},\dots,
\fl{\tfrac {m_d}p}p^{s+1}\right)\right)
\overline W_s(\bm)\Bigg).
\end{multline}
\end{lem}

\begin{proof}
Let 
$$
X_s =
\sum _{\bnull\le\bm\le(p^{r-s}-1)\bid}
Z(m_1p^s,\dots,m_dp^s)\overline W_s(\bm)
$$
and 
$$
Y_s = 
\sum _{\bnull\le\bm\le(p^{r-s}-1)\bid}
Z\left(\fl{\tfrac {m_1}p}p^{s+1},\dots,
\fl{\tfrac {m_d}p}p^{s+1}\right)
\overline W_s(\bm).
$$

By definition, we have 
$$
X_s= 
\sum _{\bnull\le\bm\le(p^{r-s}-1)\bid}\Bigg(
\sum_{p^s\bm\le\bk\le p^s(\bm+\bid)-\bid}^{} 
Z(m_1p^s,\dots,m_dp^s) W(k_1, \dots, k_d)\Bigg).
$$
For $k_j\in \big\{m_jp^s, \dots, (m_j+1)p^s-1 \big\}$,  
we have $m_j=\fl{k_j/p^s}$, $j=1, \dots, d$, 
and furthermore we have the partition 
$$
\big\{0, \dots ,p^r-1\big\}^d= 
\bigcup_{\bnull\le\bm\le(p^{r-s}-1)\bid}^{}
\prod_{j=1}^d\big\{m_jp^s, \dots, (m_j+1)p^s-1\big\}.
$$
Hence, it follows that 
$$
X_s = 
\sum _{\bnull\le\bk\le(p^{r}-1)\bid}
Z\left(\fl{\frac{k_1}{p^s}}p^s, \dots, \fl{\frac{k_d}{p^s}}p^s\right)W(k_1, \dots, k_d).
$$

Similarly, we have 
$$
Y_s=
\sum _{\bnull\le\bk\le(p^{r}-1)\bid}
Z\left(\fl{\frac{k_1}{p^{s+1}}}p^{s+1}, \dots, \fl{\frac{k_d}{p^{s+1}}}p^{s+1}\right)W(k_1, \dots,
k_d),
$$
where we used that  
$\fl{\frac1p \fl{\frac{k}{p^s}}}= \fl{\frac{k}{p^{s+1}}}$.
We therefore have
\begin{align*}
\sum_{s=0}^{r-1} (X_s-Y_s)&= 
\sum _{\bnull\le\bk\le(p^{r}-1)\bid}
W(k_1, \dots, k_d)  
\\
&\quad \times \sum_{s=0}^{r-1}\left( Z\left(\fl{\frac{k_1}{p^s}}p^s, \dots, \fl{\frac{k_d}{p^s}}p^s\right)
-   
Z\left(\fl{\frac{k_1}{p^{s+1}}}p^{s+1}, \dots, \fl{\frac{k_d}{p^{s+1}}}p^{s+1}\right)
 \right)
\\
& = 
\sum _{\bnull\le\bk\le(p^{r}-1)\bid}
W(\bk) \big( Z(\bk)-Z(\bnull) \big),
\end{align*}
because the sum over $s$ is a telescoping sum. Since 
$$
\sum _{\bnull\le\bk\le(p^{r}-1)\bid}
W(\bk) = \overline{W}_r(\bnull), 
$$
this completes the proof of the lemma.
\end{proof}

\section{Proof of Theorem~\ref{theo:1}}\label{sec:congruence}

We adapt Dwork's proof \cite[Theorem 1.1]{dwork} 
of the special case $d=1$, that is, the case 
in which there is just one variable.

For integer vectors $\bk,\bK,\bv\in\mathbb{Z}$ with 
$\bk\ge \bnull$ and $0\le v_i<p$ for $i=1,2,\dots,d$, set
$$
U(\bk,\bK) = 
A(\bv +p(\bK-\bk))A(\bk) - A(\bv +p\bk)A(\bK-\bk),
$$
being $0$ if $k_i>K_i$ for some $i$, which is the case in particular 
if $K_i<0$ for some $i$.
Furthermore, for a vector $\bm\in\Z^d$ with $\bm\ge \bnull$, set
\begin{equation} \label{eq:Hdef}
H(\bm,\bK; s )
= 
\sum_{p^s\bm\le\bk\le p^s(\bm+\bid)-\bid}
U(\bk,\bK),
\end{equation}
being $0$ if $K_i<0$ for some $i$. (The reader should recall that, by
definition, $\bid$ is the all $1$ vector.)
We omit to indicate the dependence on $p$ and $\bv$ in order to
not overload notation.

\begin{lem} \label{lem:3}
Let $\bk,\bK,\bv\in\mathbb{Z}$ with 
$\bk\ge \bnull$ and $0\le v_i<p$ for $i=1,2,\dots,d$. Then there hold
the following three facts:
\begin{enumerate}
\item[$(i)$]  We have $U(\bK-\bk, \bK) = - U(\bk,\bK)$.
\item[$(ii)$] 
For all integer vectors $\bM$ with $p^{s+1}(\bM+\bid)>\bK$, we have
$$
\sum_{\bnull\le\bm\le\bM} H(\bm,\bK; s) = 0.
$$
\item[$(iii)$] We have
$$
H(\bk,\bK; s+1 ) = \sum_{\bnull\le \bi\le(p-1)\bid}
H(\bi+p\bk,\bK; s ).
$$
\end{enumerate}
\end{lem}
\begin{proof} The assertion $(i)$ is obvious.

\medskip 

$(ii)$ We have 
\begin{align*}
\sum_{\bnull\le\bm\le\bM} H(\bm,\bK; s ) 
&= 
\sum_{\bnull\le\bm\le\bM}\Bigg(
\sum_{p^s\bm\le\bk\le p^s(\bm+\bid)-\bid} U(\bk,\bK)\Bigg)
\\
&= \sum_{\bnull\le\bk \le p^s(\bM+\bid)-\bid}
U(\bk,\bK)
\\
&= \sum_{\bnull\le\bk\le\bK} U(\bk,\bK)
\\
&=0.
\end{align*}
Here, in order to pass from the second to the third line, we used the
fact that $U(\bk,\bK)=0$ if $k_i>K_i$ for some $i$ between $1$ and $d$. 
To obtain the last
line, we used the functional equation given in $(i)$.

\medskip

$(iii)$ We have 
$$
\sum_{\bnull\le\bi\le(p-1)\bid}
H(\bi+p\bk,\bK; s) = 
\sum_{\bnull\le\bi\le(p-1)\bid}\Bigg(
\sum_{p^s(\bi+p\bm)\le\bk\le p^s(\bi+p\bm+\bid)-\bid} U(\bk,\bK)\Bigg),
$$
and it is rather straightforward to see that this sum
simply equals
$
H(\bm,\bK; s +1).
$
\end{proof}

\begin{proof}[Proof of Theorem~\ref{theo:1}]

We define two assertions, denoted by $\al_s$ and $\beta_{t,s}$, in
the following way:
for all $s\ge 0$, $\al_s$ is the assertion that the congruence 
$$
H(\bm,\bK; s ) \equiv 0 \mod  p^{s+1} g(\bm) \mathbb{Z}_p
$$
holds for all vectors $\bm,\bK\in\Z^d$ with $\bm\ge \bnull$.

For all integers $s$ and $t$ with $0\le t\le s$, $\beta_{t,s}$ is the assertion
that the congruence 
\begin{equation} \label{eq:bts}
H(\bm,\bK+p^s\bm; s ) \equiv  
\sum_{\bnull\le\bk\le(p^{s-t}-1)\bid} 
\frac{A(\bk+p^{s-t}\bm)}{A(\bk)} 
\,H(\bk,\bK; t )
\mod p^{s+1}g(\bm)\mathbb{Z}_p
\end{equation}
holds for all vectors $\bm,\bK\in\Z^d$ with $\bm\ge 0$.

Moreover, we define three further assertions $A1, A2, A3$:

$A1$: for all vectors $\bk,\bK\in\Z^d$ with $\bk\ge 0$,
we have $U(\bk,\bK ) \in pg(\bk)\mathbb{Z}_p.$

\medskip

$A2$: for all vectors $\bm,\bk,\bK\in\Z^d$ and integers $s\ge0$ with
$\bm\ge\bnull$ and $0\le k_i<p^s$ for $i=1,2,\dots,d$, we have 
$$
U(\bk+p^s\bm,\bK+p^s\bm ) \equiv \frac{A(\bk+p^s\bm)}{A(\bk)} 
U(\bk,\bK ) \mod p^{s+1}g(\bm)\mathbb{Z}_p.
$$

\medskip

$A3$: for all integers $s$ and $t$ with $0\le t<s$, 
we have 
$$\text {``$\al_{s-1}$ and $\be_{t,s}$ together imply $\be_{t+1,s}$.''}
$$

In the following, we shall first show that Assertions~$A1$, $A2$, $A3$
altogether imply Theorem~\ref{theo:1}, see the ``first step" below.
Subsequently, in the ``second step," we show that Assertions~$A1$, 
$A2$, $A3$ hold indeed.

\medskip
{\sc First step.} We claim that Theorem~\ref{theo:1} follows from
$A1$, $A2$ and $A3$.
So, from now on we shall assume that $A1$, $A2$ and $A3$ are true.
Our goal is to show that $\al_{s}$ holds for all $s\ge 0$.
We shall accomplish this by induction on $s\ge 0$. 

We begin by establishing $\al_0$. To do so, we observe that
\begin{equation} \label{eq:H=U}
H(\bm,\bK; 0 )=U(\bm,\bK),
\end{equation}
that is, that Assertion~$\al_0$ is equivalent to $A1$. 
Hence, Assertion $\al_0$ is true.

We now suppose that $\al_{s-1}$ is true. 
We shall show by induction on $t\ge 0$ that $\beta_{t,s}$ is
true for all $t\le s.$ Because of $A3$, it suffices to prove that
$\beta_{0,s}$ is true. To do so, we see that
\begin{align} \notag
\sum_{\bnull\le\bk\le(p^s-1)\bid} \frac{A(\bk+p^s\bm)}{A(\bk)}&
\,H(\bk,\bK; 0 )\\
\notag
& =\sum_{\bnull\le\bk\le(p^s-1)\bid} \frac{A(\bk+p^s\bm)}{A(\bk)}
\,U(\bk,\bK) 
\\
\notag
&\equiv \sum_{\bnull\le\bk\le(p^s-1)\bid}
U(\bk+p^s\bm,\bK+p^s\bm ) \mod p^{s+1}g(\bm)\mathbb{Z}_p
\\
&\equiv H(\bm,\bK+p^s\bm; s ) \mod p^{s+1}g(\bm)\mathbb{Z}_p.
\label{eq:Hcong}
\end{align}
Here, the first equality results from \eqref{eq:H=U}, the subsequent
congruence results from $A2$, and the last line is obtained by
remembering the definition \eqref{eq:Hdef} of $H$
(there holds in fact equality between the last two lines).
The congruence \eqref{eq:Hcong} is nothing else but Assertion~$\beta_{0,s}$, 
which is therefore proved under our assumptions.

The above argument shows in particular that
Assertion~$\beta_{s,s}$ is true, which means that we have
the congruence 
\begin{equation}\label{eq:betass}
H(\bm,\bK+p^s\bm; s ) 
\equiv \frac{A(\bm)}{A(\bnull)}\,H(\bnull,\bK; s ) 
\mod p^{s+1}g(\bm)\mathbb{Z}_p.
\end{equation}

Let us now consider the property $\gamma_{\bK}$ defined by
$$\text {$\gamma_{\bK}$: \quad 
$H(\bnull,\bK; s ) \equiv 0 
\mod p^{s+1}\mathbb{Z}_p$.}
$$ 
This property holds certainly if $K_i<0$ for some $i$
because in that case each term of the multiple sum that defines $H$ vanishes. 
We want to show that the assertion also holds when $\bK\ge\bnull$.
Let $\bK'$ be one of the vectors of non-negative integers (if there
is at all) such that $\v{\bK'}=K'_1+K'_2+\dots+K'_d$ is minimal and 
$\gamma_{\bK'}$ does {\it not\/} hold. 
Let $\bm\in\Z^d$ be a vector with $\bm \ge \bnull$ and $\v\bm>0$, 
and set $\bK=\bK'-p^s\bm$.  Since $\v\bK<\v\bK'$, we have
$$
H(\bnull,\bK; s ) \equiv 0 
\mod p^{s+1}\mathbb{Z}_p
$$
because $\gamma_{\bK}$ holds by minimality of $\bK'$. 
Since $A(\bm)/A(\bnull)\in\mathbb{Z}_p$ by Properties~$(i)$
and $(ii)$ in the statement of Theorem~\ref{theo:1}, 
it follows from~\eqref{eq:betass} that
\begin{equation}\label{eq:HKJprime}
H(\bm,\bK'; s ) 
\equiv 0
\mod p^{s+1}\mathbb{Z}_p
\end{equation}
provided $\bm \ge \bnull$ et $\v\bm>0$. 

However, by Lemma~\ref{lem:3}, $(ii)$, we know that
$$
\sum_{\bnull\le\bm\le\bM} H(\bm,\bK'; s ) =0
$$
if one chooses $\bM$ sufficiently large. Isolating the term
$H(\bnull,\bK'; s )$, this equation can be rewritten as
$$
H(\bnull,\bK'; s )=-\underset{\v\bm>0}{\sum_{\bnull\le \bm\le \bM}} 
H(\bm,\bK'; s ).
$$
The sum on the right-hand side is congruent to
$0$~mod~$p^{s+1}$ by~\eqref{eq:HKJprime}, whence
$$
H(\bnull,\bK'; s ) \equiv 0 \mod p^{s+1}.
$$
This means that $\gamma_{\bK'}$ is true, which is absurd.
Assertion~$\gamma_{\bK}$ is therefore true for all 
$\bK\in \mathbb{Z}^d$.

Let us now return to Assertion~$\be_{s,s}$, which is displayed
explicitly in \eqref{eq:betass}.
We have just shown that 
$H(\bnull,\bK; s ) \equiv 0 \mod p^{s+1}$, while $A(\bm)/A(\bnull)\in 
g(\bm)\mathbb{Z}_p$ by Properties~$(i)$ and $(ii)$ in the statement of
Theorem~\ref{theo:1}. Hence, we have also
$$
H(\bm,\bK+p^s\bm; s ) 
\equiv 0
\mod p^{s+1}g(\bm)\mathbb{Z}_p.
$$ 
By replacing $\bK$ by $\bK-p^s\bm $ (which is possible
because $\bK$ can be chosen freely from $\mathbb{Z}^d$), we see that
this is nothing else but Assertion~$\al_{s}$. 
Thus, Theorem~\ref{theo:1} follows indeed from
the truth of $A1, A2$ and $A3$.

\medskip

{\sc Second step.} 
It remains to prove Assertions~$A1$, $A2$ and $A3$ themselves, 
which we shall do in this order.

\smallskip
{\sl Proof of $A1$.} The assertion holds if $k_i>K_i$ or if $K_i<0$ 
for some $i$.
If $\bK\ge\bk\ge \bnull$, we have 
\begin{multline*}
U(\bk,\bK) = 
A(\bK-\bk)A(\bv )\left(\frac{A(\bv +p\bk)}{A(\bv )}-\frac{A(\bk)}{A(\bnull)}\right) 
\\
+
A(\bk)A(\bv )\left(\frac{A(\bK-\bk)}{A(\bnull)} - \frac{A(\bv
+p(\bK-\bk))}{A(\bv )}\right) .
\end{multline*}
Property~$(iii)$ in the statement of Theorem~\ref{theo:1} with
$\bu=\bnull$, $\bn=\bk$, $s=0$ says that
$$
\frac{A(\bv +p\bk)}{A(\bv )}-\frac{A(\bk)}{A(\bnull)}\in p \,\frac{g(\bk)}{g(\bv )}\,\mathbb{Z}_p
$$
while its special case in which $\bu=\bnull$, $\bn=\bK-\bk$, $s=0$
reads
$$
\frac{A(\bK-\bk)}{A(\bnull)} - \frac{A(\bv +p(\bK-\bk))}{A(\bv )}
\in p \,\frac{g(\bK-\bk)}{g(\bv )}\,\mathbb{Z}_p.
$$
Hence,
\begin{equation*}
A(\bK-\bk)A(\bv )\left(\frac{A(\bv +p\bk)}{A(\bv )}-\frac{A(\bk)}{A(\bnull)}\right) 
\in p \,g(\bk)A(\bK-\bk)\frac{A(\bv )}{g(\bv )}\,\mathbb{Z}_p \subseteq p g(\bk)\,\mathbb{Z}_p
\end{equation*}
and
\begin{equation*}
A(\bk)A(\bv )\left(\frac{A(\bK-\bk)}{A(\bnull)} - \frac{A(\bv
+p(\bK-\bk))}{A(\bv )}\right)
\\
\in p g(\bk)g(\bK-\bk)\frac{A(\bk)}{g(\bk)}\frac{A(\bv )}{g(\bv )}\,\mathbb{Z}_p 
\subseteq p g(\bk)\,\mathbb{Z}_p,
\end{equation*}
where the inclusion relations result from Property~$(ii)$ in the statement
of Theorem~\ref{theo:1}. 
It therefore follows that
$$
U(\bk,\bK) \in  pg(\bk)\mathbb{Z}_p,
$$
which proves Assertion~$A1$.

\smallskip
{\sl Proof of $A2$.} By a straightforward calculation, we have
\begin{multline*}
U(\bk+p^s\bm,\bK+p^s\bm) -\frac{A(\bk+p^s\bm)}{A(\bk)} 
U(\bk,\bK) 
\\
=-A(\bK-\bk)A(\bv +p\bk)
\left(\frac{A(\bv +p\bk+p^{s+1}\bm)}{A(\bv +p\bk)}-
\frac{A(\bk+p^s\bm)}{A(\bk)}
\right).
\end{multline*}
If $K_i<0$ for some $i$, the right-hand side is zero since $A(\bK-\bk)=0$,
whence Assertion~$A2$ is trivially true. If $\bK\ge\bnull$, by
Properties~$(iii)$ and $(ii)$ in the statement of Theorem~\ref{theo:1}, 
the right-hand side is an element of
$$
 A(\bK-\bk)A(\bv +p\bk)\frac{g(\bm)}{g(\bv +p\bk)}\,p^{s+1}\mathbb{Z}_p\subseteq g(\bm)p^{s+1}\mathbb{Z}_p,
$$
which proves Assertion~$A2$ in this case as well.

\smallskip
{\sl Proof of $A3$.} Let $0\le t<s$, and
assume that $\al_{s-1}$ and $\beta_{t,s}$ are true. Under these
assumptions, we must deduce the truth of Assertion~$\beta_{t+1,s}.$

In the assertion $\beta_{t,s}$, we 
replace the summation index $\bk$ in the sum on the right-hand side
of \eqref{eq:bts} by
$\bi+p\bu$, where $0\le i_\ell<p-1$ and $0\le u_\ell<p^{s-t-1}$ for
$\ell=1,2,\dots,d$. Thus, we obtain that 
\begin{multline} \label{eq:bets}
H(\bm,\bK+p^s\bm; s ) 
\equiv  
\sum_{\bnull\le\bi\le(p-1)\bid}\Bigg(
\sum_{\bnull\le\bu\le(p^{s-t-1}-1)\bid}
\frac{A(\bi+p\bu+p^{s-t}\bm)}{A(\bi+p\bu)} 
\,H(\bi+p\bu,\bK; t )\Bigg)
\\ 
\mod p^{s+1}g(\bm)\mathbb{Z}_p.
\end{multline}
Define
$$
X:=H(\bm,\bK+p^s\bm; s ) -
\sum_{\bnull\le\bu\le(p^{s-t-1}-1)\bid}
\frac{A(\bu+p^{s-t-1}\bm )}{A(\bu)} 
\sum_{\bnull\le\bi\le(p-1)\bid} H(\bi+p\bu,\bK; t ).
$$
Since $\beta_{t,s}$ (in the form \eqref{eq:bets}) is true, we have 
\begin{multline*}
X\equiv \sum_{\bnull\le\bi\le(p-1)\bid}\Bigg(
\sum_{\bnull\le\bu\le(p^{s-t-1}-1)\bid}
H(\bi+p\bu,\bK; t )
\\
\times
\left(\frac{A(\bi+p\bu+p^{s-t}\bm)}{A(\bi+p\bu)}
- \frac{A(\bu+p^{s-t-1}\bm )}{A(\bu)} \right) \Bigg)
\mod p^{s+1}g(\bm)\,\mathbb{Z}_p.
\end{multline*}
Since $u_i<p^{s-t-1}$ for all $i$, Property~$(iii)$ in the statement
of Theorem~\ref{theo:1} implies that
\begin{equation}\label{eq:AA-AA}
\frac{A(\bi+p\bu+p^{s-t}\bm)}{A(\bi+p\bu)}
- \frac{A(\bu+p^{s-t-1}\bm )}{A(\bu)}
\in 
p^{s+1}\frac{g(\bm)}{g(\bi+p\bu)}\,\mathbb{Z}_p.
\end{equation}
Moreover, since $t<s$, Assertion~$\al_{s-1}$ implies that
\begin{equation}\label{eq:Ht<s}
H(\bi+p\bu,\bK; t ) \in p^{t+1}g(\bi+p\bu)\mathbb{Z}_p.
\end{equation}
It now follows from~\eqref{eq:AA-AA} and~\eqref{eq:Ht<s} that 
$X\equiv 0 \mod p^{s+1}g(\bm)\mathbb{Z}_p.$

However, by Lemma~\ref{lem:3}, $(iii)$, we know that
$$
\sum_{\bnull\le\bi\le(p-1)\bid} H(\bi+p\bu,\bK; t ) =  
H(\bu,\bK; t+1 ),
$$
which can be used to simplify $X$ to
$$
X= H(\bm,\bK+p^s\bm; s ) -
\sum_{\bnull\le\bu\le(p^{s-t-1}-1)\bid}
\frac{A(\bu+p^{s-t-1}\bm )}{A(\bu)} 
H(\bu,\bK; t+1 ).
$$
Since $X\equiv 0 \mod p^{s+1} g(\bm)\mathbb{Z}_p,$ the preceding identity
shows that
$$
H(\bm,\bK+p^s\bm; s ) 
\equiv 
\sum_{\bnull\le\bu\le(p^{s-t-1}-1)\bid}
\frac{A(\bu+p^{s-t-1}\bm )}{A(\bu)} 
H(\bu,\bK; t+1 ) 
\mod p^{s+1} g(\bm)\mathbb{Z}_p.
$$
This is nothing else but Assertion~$\beta_{t+1,s}$.
Hence, Assertion $A3$ is established.

\medskip
This completes the proof of Theorem~\ref{theo:1}.
\end{proof}

\section{Theorem~\ref{theo:1} implies Theorem~\ref{conj:1}}\label{sec:reduction}

We want to prove that Theorem~\ref{theo:1} can be applied for
$A=g=B_\bN.$ 
In order to see this, we first establish some intermediary lemmas, 
extending corresponding auxiliary results in Section~7 of \cite{kr}
to higher dimensions.

\begin{lem} \label{lem:1} 
Under the assumptions of Theorem~\ref{theo:1}, we have
$$
\frac{B_\bN(\bv +p\bu+p^{s+1}\bn )}{B_\bN(p\bu+p^{s+1}\bn )}
= \frac{B_\bN(\bv +p\bu)}{B_\bN(p\bu)} +
\mathcal{O}\big(p^{s+1}\big),
$$
where $\mathcal{O}(R)$ denotes an element of $R\mathbb Z_p$.
\end{lem}
\begin{proof}
Recalling the definition of $B(\bN^{(j)},\bm)$ in \eqref{eq:Bj},
we have
\begin{align*}
&\frac{B(\bN^{(j)},\bv +p\bu+p^{s+1}\bn )}{B(\bN^{(j)},p\bu+p^{s+1}\bn )} 
\\
&\kern1cm
=
\frac{\left(
\sum _{i=1} ^{d}N^{(j)}_i(pu_i+p^{s+1}n_i)+\sum _{i=1} ^{d}N^{(j)}_i
v_i\right)\cdots \left(
\sum _{i=1} ^{d}N^{(j)}_i(pu_i+p^{s+1}n_i)+1\right)}
{
\displaystyle\prod _{i=1} ^{d}
\Big((v_i +pu_i+p^{s+1}n_i)\cdots (1 +pu_i+p^{s+1}n_i)
\Big)^{N^{(j)}_i}}
\\
&\kern1cm
=
\frac{\left(
\sum _{i=1} ^{d}N^{(j)}_i(pu_i)+\sum _{i=1} ^{d}N^{(j)}_i
v_i\right)\cdots \left(
\sum _{i=1} ^{d}N^{(j)}_i(pu_i)+1\right)+\mathcal{O}(p^{s+1})}
{
\displaystyle\prod _{i=1} ^{d}
\Big((v_i +pu_i)\cdots (1 +pu_i)
\Big)^{N^{(j)}_i}+\mathcal{O}(p^{s+1})}.
\end{align*}
We claim that this implies
\begin{align*}
&\frac{B(\bN^{(j)},\bv +p\bu+p^{s+1}\bn )}{B(\bN^{(j)},p\bu+p^{s+1}\bn )} 
\\
&\kern1cm
=
\frac{\left(
\sum _{i=1} ^{d}N^{(j)}_i(pu_i)+\sum _{i=1} ^{d}N^{(j)}_i
v_i\right)\cdots \left(
\sum _{i=1} ^{d}N^{(j)}_i(pu_i)+1\right)}
{
\displaystyle\prod _{i=1} ^{d}
\Big((v_i +pu_i)\cdots (1 +pu_i)
\Big)^{N^{(j)}_i}}+\mathcal{O}(p^{s+1})
\\
&\kern1cm
= \frac{B(\bN^{(j)},\bv +p\bu)}{B(\bN^{(j)},p\bu)} + \mathcal{O}\big(p^{s+1}\big).
\end{align*}
Indeed, if $\bv=\bnull$, then this holds trivially.  
If $\bv>\bnull$, then, together with the hypothesis $ v_i<p$, 
we infer that $(v_i+pu_i)(v_i+pu_i-1)\cdots (1+pu_i)$  
is not divisible by $p$, which implies in particular that 
$B(\bN^{(j)},\bv+p\bu)/B(\bN^{(j)},p\bu)\in \mathbb{Z}_p$. This allows
us to arrive at the above
conclusion in the same style as in Section~7.1 in \cite{kr}.

By taking products, we deduce
$$
\prod_{j=1}^k
\frac{B(\bN^{(j)},\bv +p\bu+p^{s+1}\bn )}{B(\bN^{(j)},p\bu+p^{s+1}\bn )} 
=\prod_{j=1}^k\left(
 \frac{B(\bN^{(j)},\bv +p\bu)}{B(\bN^{(j)},p\bu)} + \mathcal{O}\big(p^{s+1}\big).
\right).
$$
By expanding the product on the right-hand side and using that 
$\displaystyle \frac{B(\bN^{(j)},\bv+p\bu)}{B(\bN^{(j)},p\bu)}\in \mathbb{Z}_p,$
we obtain the assertion of the lemma.
\end{proof}

For the proof of Lemma~\ref{lem:2} below, 
we will use the $p$-adic gamma function, which is defined on integers $n\ge 1$ by 
$$
\Gamma_p(n) = (-1)^n \prod_{\stackrel{k=1}{(k,p)=1}}^{n-1} k.
$$
In the following lemma, we collect some facts about $\Gamma_p$.

\begin{lem}\label{lem:gammap}
$(i)$ For all integers $n\ge 1$, we have 
\begin{equation*}
\frac{(np)!}{n!} = (-1)^{np+1}p^n \Gamma_p(1+np).
\end{equation*}
$(ii)$ For all integers $k\ge 1,n\ge 1,s\ge 0$, we have 
\begin{equation*}
\Gamma_p(k+np^s) \equiv \Gamma_p(k) \mod p^s.
\end{equation*}
\end{lem}

The above two properties of the $p$-adic gamma function 
are now used in the proof of the following result.

\begin{lem} \label{lem:2} We have
$$
\frac{B_\bN(p\bu+p^{s+1}\bn )}{B_\bN(\bu+p^{s}\bn )}
=\frac{B_\bN(p\bu)}{B_\bN(\bu)} \big(1+ \mathcal{O}(p^{s+1})\big).
$$
\end{lem}
\begin{proof} We have
\begin{align} 
\frac{B(\bN^{(j)},p\bu+p^{s+1}\bn )}{B(\bN^{(j)},\bu+p^{s}\bn )} 
&= (-1)^{1+\v{\bN^{(j)}}}\frac{\Gamma_p\big(1+
\bN^{(j)}\cdot(p\bu+p^{s+1}\bn)\big)}
{
\prod _{i=1} ^{d}
\Gamma_p\big(1+pu_i+p^{s+1}n_i\big)^{N^{(j)}_i}}
\label{eq:firstequality} \\
&= (-1)^{1+\v{\bN^{(j)}}}\frac{\Gamma_p\big(1+
p\bN^{(j)}\cdot \bu\big) + \mathcal{O}\big(p^{s+1}\big)}
{
\prod _{i=1} ^{d}
\Gamma_p\big(1+pu_i\big)^{N^{(j)}_i} + \mathcal{O}\big(p^{s+1}\big)}
\label{eq:rajoutcorrectionbis} \\
&= (-1)^{1+\v{\bN^{(j)}}}\frac{\Gamma_p\big(1+
p\bN^{(j)}\cdot \bu\big)}
{
\prod _{i=1} ^{d}
\Gamma_p\big(1+pu_i\big)^{N^{(j)}_i} }
\big(1+\mathcal{O}(p^{s+1})\big)
 \label{eq:rajoutcorrection}\\
&= \frac{B(\bN^{(j)},p\bu)}{B(\bN^{(j)},\bu)} \big(1+ \mathcal{O}(p^{s+1})\big).
\label{eq:2}
\end{align}
where $(i)$ of Lemma~\ref{lem:gammap} is used 
to see~\eqref{eq:firstequality} and~\eqref{eq:2}, and 
$(ii)$ is used for~\eqref{eq:rajoutcorrectionbis}.
Equation~\eqref{eq:rajoutcorrection} holds because $\Gamma_p(1+pu_i)$
and $\Gamma_p(1+p\bN^{(j)}\cdot\bu)$  
are both not divisible by $p$. Taking the product over $j=1, 2, \dots,
k$, we obtain the assertion of the lemma.
\end{proof}

Before 
proceeding, we remark that 
$v_p\big(B(\bN^{(j)},p^s\bu)/B(\bN^{(j)},\bu)\big)=0$ for any integer $s\ge 0$, 
which can be proved in the same way as Lemma~13 in~\cite{kr}. This property will be used twice below.

We now multiply both sides of the congruences obtained in Lemmas~\ref{lem:1} and~\ref{lem:2}.
Thus, we obtain
\begin{align*}
\frac{B_\bN(\bv +p\bu+\bn p^{s+1})}{B_\bN(\bu+\bn p^{s})} 
&= \frac{B_\bN(\bv +p\bu)}{B_\bN(\bu)} \big( 1+\mathcal{O}(p^{s+1})\big) + 
 \frac{B_\bN(p\bu)}{B_\bN(\bu)} \mathcal{O}\big(p^{s+1}\big)
\\
&= \frac{B_\bN(\bv +p\bu)}{B_\bN(\bu)} \big( 1+\mathcal{O}(p^{s+1})\big) + 
 \mathcal{O}\big(p^{s+1}\big)
\end{align*}
(since $v_p\big(B_\bN(p\bu)/B_\bN(\bu)\big)=0$ by the remark above), 
which, in its turn, can be rewritten as
\begin{equation*}
\frac{B_\bN(\bv +p\bu+\bn p^{s+1})}{B_\bN(\bv +p\bu)} = 
\frac{B_\bN(\bu+\bn p^{s})}{B_\bN(\bu)} 
+ 
\frac{B_\bN(\bu+\bn p^{s})}{B_\bN(\bu)}\mathcal{O}\big(p^{s+1}\big) +
\frac{B_\bN(\bu+\bn p^{s})}{B_\bN(\bv +p\bu)}\mathcal{O}\big(p^{s+1}\big).
\end{equation*}

It remains to show that
\begin{equation}\label{eq:fact1}
\frac{B_\bN(\bu+\bn p^{s})}{B_\bN(\bu)}
\in \frac{B_\bN(\bn)}{B_\bN(\bv +p\bu)}\,\mathbb{Z}_p
\end{equation}
and
\begin{equation}\label{eq:fact2}
\frac{B_\bN(\bu+\bn p^{s})}{B_\bN(\bv +p\bu)}
\in \frac{B_\bN(\bn)}{B_\bN(\bv +p\bu)}\,\mathbb{Z}_p.
\end{equation}
These two facts will follow from the next lemma.

\begin{lem} \label{lem:33}
For all non-negative integers $s$, all integer vectors $\bn\in\Z^d$
with $\bn\ge \bnull$, and all integer vectors $\bu\in\Z^d$ with $0\le
u_i<p^s$, $i=1,2,\dots,d$, we have 
$$
\frac{B_\bN(\bu+\bn p^s)}{B_\bN(\bu)} \in B_\bN(\bn)\mathbb{Z}_p.
$$
\end{lem}
\begin{proof} We have
$$
\frac{B(\bN^{(j)}, \bu+\bn p^s)}{B(\bN^{(j)},\bu)} = 
\frac{\displaystyle\binom{\scriptstyle
\sum _{i=1} ^{d}N^{(j)}_i(u_i+n_ip^s)}{\scriptstyle
\sum _{i=1} ^{d}N^{(j)}_iu_i}}
{
\prod _{i=1} ^{d}\binom{u_i+n_ip^s}{u_i}^{N^{(j)}_i}}
\cdot \frac {B(\bN^{(j)},\bn p^s)}{B(\bN^{(j)},\bn)}\cdot B(\bN^{(j)},\bn).
$$
On the right-hand side, the term $B(\bN^{(j)},\bn p^s)/B(\bN^{(j)},\bn)$
and the binomial coefficients
$\binom{u_i+n_ip^s}{u_i}$ have vanishing $p$-adic valuation 
(this has already been observed in the paragraph after the end of the
proof of Lemma~\ref{lem:2}). Thus we have
\begin{equation}\label{eq:lem33}
\frac{B(\bN^{(j)},\bu+\bn p^s)}{B(\bN^{(j)},\bu)} \in  B(\bN^{(j)},\bn)\mathbb{Z}_p.
\end{equation}
The lemma follows by taking the product over $j\in\{1, \ldots, k\}$ of both 
sides of~\eqref{eq:lem33}.
\end{proof}

The preceding lemma implies
$$
\frac{B_\bN(\bu+\bn p^{s})}{B_\bN(\bu)} \in B_\bN(\bn)\mathbb{Z}_p 
\subseteq \frac{B_\bN(\bn)}{B_\bN(\bv +p\bu)}\,\mathbb{Z}_p,
$$
which proves~\eqref{eq:fact1}. 
Moreover, still due to Lemma~\ref{lem:33}, we have
\begin{multline*}
\frac{B_\bN(\bu+\bn p^{s})}{B_\bN(\bv +p\bu)} = 
\frac{B_\bN(\bu+\bn p^{s})}{B_\bN(\bu)}\cdot B_\bN(\bu)\cdot 
\frac{1}{B_\bN(\bv +p\bu)} 
\\
\in B_\bN(\bu)\cdot \frac{B_\bN(\bn)}{B_\bN(\bv +p\bu)} \,
\mathbb{Z}_p 
\subseteq \frac{B_\bN(\bn)}{B_\bN(\bv +p\bu)}\,\mathbb{Z}_p,
\end{multline*}
which proves~\eqref{eq:fact2}.  
Therefore,
$$
\frac{B_\bN(\bv +p\bu+\bn p^{s+1})}{B_\bN(\bv +p\bu)}
- \frac{B_\bN(\bu+\bn p^s)}{B_\bN(\bu)}
\in p^{s+1} \,\frac{B_\bN(\bn)}{B_\bN(\bv +p\bu)}\,\mathbb{Z}_p,
$$
which shows that Property~$(iii)$ of Theorem~\ref{theo:1} is
satisfied. Since Properties~$(i)$ and $(ii)$ are trivially true, we
can hence apply the latter theorem.

\section{Proof of Lemma~\ref{lem:11}}\label{sec:lem11}

The claim is trivially true if $p$ divides $m_i$ for all $i$.
We may therefore assume that $p$ does not divide $m_i$ for some $i$
between $1$ and $d$ for the rest of the proof.
Let us write $\bm=\ba+p\bj$, with $0\le a_i<p$ for all $i$ 
(but at least one $a_i$ is positive). 
We are apparently in a similar situation as in Lemma~\ref{lem:333}.
Indeed, we may derive Lemma~\ref{lem:11} from Lemma~\ref{lem:333}. 
In order to see this, we observe that
\begin{align*}
H_{
\sum _{i=1} ^{d}L_im_ip^s}-H_{
\sum _{i=1} ^{d}L_i\fl{\frac {m_i}p}p^{s+1}}&
=\sum _{\ep=1} ^{p^s\bL\cdot\ba}\frac {1} {p^{s+1}\bL\cdot\bj+\ep}\\
&=
\sum _{\ep=1} ^{\fl{\bL\cdot\ba/p}}\frac {1} {p^{s+1}\bL\cdot\bj+p^{s+1}\ep}
+
\underset{p^{s+1}\nmid \ep}{\sum _{\ep=1} ^{p^s\bL\cdot\ba}}\frac {1}
{p^{s+1}\bL\cdot\bj+\ep}\\
&=\frac {1}
{p^{s+1}}(H_{\bL\cdot\bj+\fl{\bL\cdot\ba/p}}-H_{\bL\cdot\bj})+
\underset{p^{s+1}\nmid \ep}{\sum _{\ep=1} ^{p^s\bL\cdot\ba}}\frac {1}
{p^{s+1}\bL\cdot\bj+\ep}.
\end{align*}
Because of $v_p(x+y)\ge\min\{v_p(x),v_p(y)\}$, this implies
$$
v_p\big(H_{
\sum _{i=1} ^{d}L_im_ip^s}-H_{
\sum _{i=1} ^{d}L_i\fl{\frac {m_i}p}p^{s+1}}\big)\ge
\min\{ -1-s+v_p(H_{\bL\cdot\bj+\fl{\bL\cdot\ba/p}}-H_{\bL\cdot\bj}),-s\}.
$$
It follows that
\begin{multline*}
v_p\Big(B_{\mathbf N}(\bm)
\big(H_{
\sum _{i=1} ^{d}L_im_ip^s}-H_{
\sum _{i=1} ^{d}L_i\fl{\frac {m_i}p}p^{s+1}}\big)\Big)\\
\ge
-1-s+\min\left\{v_p\Big(B_{\mathbf N}(\ba+p\bj)
(H_{\bL\cdot\bj+\fl{\bL\cdot\ba/p}}-H_{\bL\cdot\bj})\big),
1+v_p\big(B_{\mathbf N}(\ba+p\bj)\Big)\right\}.
\end{multline*}
Use of Lemma~\ref{lem:333} then completes the proof.
\quad \quad \qed


\begin{thebibliography}{1}\label{sec:biblio}
\addcontentsline{toc}{section}{References}

\bibitem{almkvist} G. Almkvist and W. Zudilin, {\em 
Differential equations, mirror maps and zeta values}, in: Mirror
Symmetry~V,  N. Yui, S.-T. Yau,  
and J.D. Lewis (eds.), AMS/IP Studies in Advanced Mathematics {\bf 38} (2007), 
International Press \& Amer. Math. Soc., 481--515.


\bibitem{enck} G. Almkvist, C. van Enckevort, D. van Straten and W. Zudilin, {\em Tables of Calabi--Yau equations}, 104 p., 2005. 
Available at {\tt http://arxiv.org/abs/math/0507430}.

\bibitem{batstrat} V. V. Batyrev and D. van Straten,
{\em Generalized hypergeometric functions and rational curves on Calabi--Yau
complete intersections in toric varieties},
Comm. Math. Phys. {\bf 168} (1995), 493--533.

\bibitem{beukers} F. Beukers, {\em
Irrationality proofs using modular forms},
Journ\'ees arithm\'etiques de Besan\c{c}on (1985).
Ast\'erisque {\bf 147}-{\bf 148}, (1987), 271--283, 345.

\bibitem{beukers2} F. Beukers, 
{\em On Dwork's accessory parameter problem}, 
Math. Z. {\bf  241} (2002), 425--444.

\bibitem{dela} E. Delaygue, {\em Crit\`ere pour l'integralit\'e des
    coefficients de Taylor des applications miroir}, preprint, 2009.

\bibitem{dworkihes} B. Dwork, 
{\em $p$-adic cycles},
Inst.\ Hautes \'Etudes Sci.\ Publ.\ Math.\ {\bf 37} (1969), 27--115. 

\bibitem{dwork1}
B. Dwork, {\em On $p$-adic differential equations. I. The Frobenius
structure of differential equations}, Table Ronde d'Analyse non
archim\'edienne (Paris, 1972), Bull.\ Soc.\ Math.\ France,
Mem. No.~39--40, Soc.\ Math.\ France, Paris, 1974, pp.~27--37. 

\bibitem{dwork2}
B. Dwork, {\em On $p$-adic differential
equations. II. The $p$-adic asymptotic behavior of solutions of
ordinary linear differential equations with rational function
coefficients}, Ann.\ Math.\ (2) {\bf 98} (1973), 366--376. 

\bibitem{dwork3}
B. Dwork, {\em On $p$-adic differential equations. III. On $p$-adically
bounded solutions of ordinary linear differential equations with
rational function coefficients}, Invent.\ Math.\ {\bf 20} (1973),
35--45. 

\bibitem{dwork} B. Dwork, {\em On $p$-adic differential equations IV:
generalized  
hypergeometric functions as $p$-adic analytic functions in one
variable}, Ann.\ Sci.\ \'E.N.S. (4) {\bf 6}, no.~3 (1973), 295--316.

\bibitem{hosono} S. Hosono, A. Klemm, S. Theisen and S.-T. Yau, 
{\em Mirror symmetry, mirror map and applications to complete intersection 
Calabi--Yau spaces}, 
Nuclear Phys. B {\bf 433}, no.~3 (1995), 501--552.

\bibitem{KoSVAA} M. Kontsevich, A. Schwarz and V. Vologodsky, 
{\em Integrality of instanton numbers and $p$-adic $B$-model},
Phys. Lett. B {\bf 637} (2006), 97--101. 

\bibitem{kr} C. Krattenthaler and T. Rivoal, {\em On the integrality of the 
Taylor coefficients of mirror maps}, Duke Math.\ J. (to appear).

\bibitem{kr2} C. Krattenthaler and T. Rivoal, {\em On the integrality of the 
Taylor coefficients of mirror maps, II}, Commun.\ Number Theory
Phys.\ (to appear).

\bibitem{lang} S. Lang, {\em Cyclotomic fields II}, Graduate Texts in
Mathematics {\bf 69}, Springer--Verlag, 1980.  

\bibitem{lianyau} B. H. Lian and S.-T. Yau, 
{\em Mirror maps, modular relations and hypergeometric series I},
appeared as {\em Integrality  
of certain exponential series}, in: Lectures in Algebra and Geometry,
Proceedings of the International Conference on Algebra and Geometry,
Taipei, 1995, M.-C.~Kang (ed.), 
Int. Press, Cambridge, MA, 1998, pp.~215--227.


\bibitem{lianyau2} B. H. Lian and S.-T. Yau, 
{\em The $n$th root of the mirror map}, in: Calabi--Yau varieties and
mirror symmetry, Proceedings of the Workshop on Arithmetic, Geometry
and Physics around Calabi--Yau Varieties and Mirror Symmetry,
Toronto, ON, 2001, 
N.~Yui and J.~D.~Lewis (eds.),
Fields Inst. Commun., {\bf 38}, Amer. Math. Soc., Providence, RI,
2003, pp.~195--199.

\bibitem{stienstra} J. Stienstra, {\em GKZ Hypergeometric Structures},
in: Arithmetic and geometry around hypergeometric functions, 
R.-P.~Holzapfel, A.~Muhammed Uluda\u g and M.~Yoshida (eds.),
Progr. Math., vol.~260, Birkh\"auser, Basel, 2007, pp.~313--371. 

\bibitem{volog} V. Vologodsky, {\em Integrality of instanton numbers},
preprint 2007; 
{\tt http://ar$\chi$iv.org/abs/0707.4617}.


\bibitem{zud} W. Zudilin, {\em Integrality of power expansions related
to hypergeometric series},  
Mathematical Notes {\bf 71}.5 (2002), 604--616.

\end{thebibliography}
\end{document}